\documentclass[10pt]{amsart}
\usepackage{amssymb,amsmath}
\usepackage{epsfig}
\usepackage{hyperref} 
\usepackage{appendix}
\usepackage[usenames]{color}
\newtheorem{thm}{Theorem}[section]
\newtheorem{remm}{Remark}[section]
\newtheorem{lem}{Lemma}[section]

\newtheorem{remark}[remm]{Remark}

\newtheorem{proposition}[thm]{Proposition}

\numberwithin{equation}{section}
\newcommand{\rset}{{\Bbb R}}

\newcommand{\expected}{{\mathbb E}}
\newcommand{\ken}{\ \ }
\newcommand{\ssy}{\scriptscriptstyle}

\newcommand{\mcA}{{\mathcal A}}

\newcommand{\mcE}{{\mathcal E}}
\newcommand{\mcF}{{\mathcal F}}

\newcommand{\mcO}{{\mathcal O}}

\newcommand{\mcS}{{\mathcal S}}

\newcommand{\half}{\frac{1}{2}}
\newcommand{\brr}{\overline}

\newcommand{\varphib}{\overline{\brr{\varphi}}}

\newcommand{\bara}{\brr{a}}
\newcommand{\barb}{\brr{b}}

\newcommand{\barra}{\overline{\brr{a}}}
\newcommand{\barrb}{\overline{\brr{b}}}

\newcommand{\Ito}{\text{\rm It{\^o}}~}
\newcommand{\tol}{\text{\rm TOL}}

\newcommand{\qhta}{\vartheta}

\newcommand{\tmax}{t_{\text{\rm max}}}
\newcommand{\ggo}{\widetilde{\gamma}_{_0}}
\newcommand{\taumax}{\tau_{\text{\rm max}}}

\newcommand{\barrg}{\overline{\overline{g}}}
\newcommand{\bbg}{\overline{\overline{g}}}
\newcommand{\bbgg}{\overline{\overline{{\widetilde g}}}}

\newcommand{\xwros}{S_{_{\Delta\tau}}}
\newcommand{\prj}{\Pi} 
\newcommand{\ksi}{\xi}
\newcommand{\OM}{\mcO}
\newcommand{\mcG}{\mathcal{G}}
\newcommand{\mcJ}{U}

\newcommand{\tlambda}{{\widetilde{\lambda}}}
\newcommand{\lstar}{\ell_{\star}}
\newcommand{\rhostar}{\rho}

\newcommand{\ErrS}{\mathop{\text{\rm E}_{_S}}}

\newcommand{\zset}{\Bbb Z}
\newcommand{\velos}{\rightarrow}
\newcommand{\kenn}{\hskip0.2truecm}

\newcommand{\emfy}{\hbox{\hbox{$\subset$}\kern-.52em\raise-.26em\hbox{$\rightarrow$}}}

\newcommand{\EES}{\mathop{\text{\tt E}_{\ssy S}}}
\newcommand{\EEtim}{\mathop{\text{\tt E}_{_{\text{\rm tim}}}}}
\newcommand{\EEtau}{\mathop{\text{\tt E}_{_{\text{\rm tau}}}}}
\newcommand{\EEtauS}{\mathop{\text{\tt E}_{_{\text{\rm tau},S}}}}
\newcommand{\EEtimS}{\mathop{\text{\tt E}_{_{\text{\rm tim},S}}}}
\begin{document}  
\title[]
{Monte Carlo Euler approximations\\
of HJM term structure financial models}
\author{T. Bj\"ork$^{+}$
\and A. Szepessy$^{\dag}$
\and R. Tempone$^{\S}$
\and G. E. Zouraris$^{\ddag}$}
%
%
\thanks{$^+$Institutionen f\"or finansiell ekonomi,
Handelsh\"ogskolan, Box 6501, S--113 83 Stockholm (Tomas.Bjork@hhs.se)}
\thanks{$^{\dag}$Matematiska Institutionen,
Kungl. Tekniska H\"ogskolan, S--100 44 Stockholm
(szepessy@kth.se)}
\thanks
{$^{\S}$ Division of Mathematics and Computational Sciences
and Engineering (MCSE),
4700 King Abdullah University of Science and Technology (KAUST),
Thuwal 23955-6900, Kingdom of Saudi Arabia
(raul.tempone@kaust.edu.sa)}
\thanks
{$^{\ddag}$ Department of Mathematics, University of Crete,
GR--714 09 Heraklion (zouraris@math.uoc.gr)}
%
%
%
%
\centerline{\today}
\begin{abstract}
We present  Monte Carlo-Euler methods for a weak
approximation problem related to the Heath-Jarrow-Morton (HJM)
term structure model, based on \Ito stochastic differential
equations in infinite dimensional spaces, and prove strong and weak
error convergence estimates. The weak error estimates
are based on stochastic flows and discrete dual backward problems,
and they can be used to identify different error contributions
arising from time and maturity discretization as well as the
classical statistical error due to finite sampling.
%
%
Explicit formulas for  efficient computation of sharp error
approximation are included. Due to the  structure of the HJM
models  considered here, the computational effort devoted to the
error estimates is low compared to the work to compute Monte Carlo
solutions to the HJM model.  Numerical examples  with known exact
solution are included in order to show the behavior of the
estimates.
\end{abstract}
\maketitle
%
%
\section{The HJM Model}\label{Section1}
\subsection{Generals}
When valuing derivatives in the bond market it is important to use
models that are consistent with the initial term structure
observed in the market. The Heath-Jarrow-Morton (HJM)
model for the forward rate has this property and in addition offers
the freedom to choose the volatility structure, for example to be
able to fit other derivative prices quoted in the market (see
\cite{Car2,Car1,Hu,Reb}). This HJM model approach is particularly
suitable for Monte Carlo computations, since in general the
alternative of tree methods leads, for the multifactor case, to
non recombining trees with higher computational cost.
\par
In this work we focus on the numerical approximation of  the price
of financial instruments in the bond market, using the HJM model
of forward rates.
%
%
%
%
We propose Monte Carlo Euler methods fow which we  develop a
rigorous strong error analysis and provide rigorous weak error
expansions, with leading error term in computable a posteriori form,
offering computational reliability in the use of more complicated HJM
multifactor models, where no explicit formula can be found, or such a
formula is just too complicated to use, for the pricing of contingent claims.
These weak error expansions can be used in adaptive
algorithms to handle {\it simultaneously} different sources of
error, e.g. time discretization, maturity discretization, and
finite sampling, see \cite{STZ}. To develop error estimates we
use a Kolmogorov backward equation in an extended domain and carry
out further the analysis in  \cite{STZ}, from general weak
approximation of \Ito~stochastic differential equations in
$\rset^n$, to weak approximation of the HJM \Ito~stochastic
differential equations in infinite dimensional
spaces. Therefore, the main new ingredient here is to provide
error estimates useful for adaptive refinement not only in time
$t$ but also in maturity time $\tau$.
In addition, using the structure of the HJM model studied here,
the application of a simple transformation  removes the error
caused by the representation of the initial term structure in a
finite maturity partition. Finally, the formulas to compute sharp
error approximations are simplified by exploiting the structure of
the HJM model, reducing the work to compute such error estimates.
The use of the error estimates proposed  here is compatible with
the application of variance reduction techniques, allowing for
faster Monte Carlo computations, see \cite{BoPh}.
\par
The work at hand is based on a research paper included in the one
of the authors PhD Disseration \cite{Tempone}.
\subsection{Description of the model}
The bond market is assumed to be efficient and without friction,
i.e. there is no arbitrage opportunity, and there exists a
martingale probability measure, under which bond contracts can be
priced as expected values of properly discounted cash flows, see
\cite{BR,Bj,Du}. On what follows, all the equations are assumed to
be under such a probability measure.
\par
The HJM model is based on the so called {\it forward rate},
$f(t,\tau)$, which relates to the price of the most simple type of
bond, the zero coupon bond, with contracting time $t$ and maturity
time $\tau$, by
\begin{equation*}
p(t,\tau)=\exp\left(-\int_t^\tau f(t,\eta)\;d\eta\right).
\end{equation*}
In particular, the non arbitrage assumption in the  HJM
formulation, see \cite{HJM1,HJM2}, yields an \Ito~ stochastic
differential equation, for $\tau\in[0,\taumax],$
\begin{equation}\label{eq:1.1}
\begin{split}
df(t,\tau)=&\,\sum_{j=1}^{\ssy J}\sigma^j(t,\tau)
\left(\int_t^{\tau}\sigma^j(t,s)ds\right)\;dt
+\sum_{j=1}^{\ssy J}\sigma^j(t,\tau)\;dW^j(t),\quad t\in[0,\tau]\\
f(0,\tau)=&\,f_{_0}(\tau).
\end{split}
\end{equation}
Here $(W^j)_{j=1}^{\ssy J}$ are independent Wiener processes, and
$(\sigma^j(t,\tau))_{j=1}^{\ssy J}$ are stochastic processes,
adapted to the filter structure generated by the Wiener processes.
Furthermore, the initial datum for the term structure,
$f_{_0}:[0,\taumax]\rightarrow\rset$, is a  deterministic function
in $C^1([0,\taumax])$. In this setting, the {\it short rate},
$r(t)$, is defined as  $r(t)\equiv f(t,t)$.
\par
On what follows the volatility function $\sigma =
(\sigma^1,\dots,\sigma^{\ssy J})$ is assumed to be of the form
\begin{equation*}
\begin{split}
\sigma(t,\tau)=&\,\ksi(r(t))\,\lambda(t,\tau)\\
=&\,\ksi(f(t,t))\,\lambda(t,\tau),
\end{split}
\end{equation*}
where $\ksi:\rset\velos\rset$ and
$\lambda:[0,\tmax]\times[0,\taumax]\rightarrow\rset^{\ssy J}$ are
given bounded functions on $C^{m_0}(\rset)$ and
$C^{m_0}([0,\tmax]\times[0,\taumax])$, respectively, for $m_0$ a
sufficiently large integer. Then, setting
$${\mathcal D}\equiv\{\,(t,\tau)\in[0,\tmax]\times[0,\taumax]:\,
t\le \tau\}$$
problem \eqref{eq:1.1} reads as
follows: find
$f=f(t,\tau):{\mathcal D}\rightarrow{\mathbb R}$
such that
\begin{equation}\label{eq:1.2a}
\begin{split}
df(t,\tau)=&\,\ksi^2(f(t,t))\,\tlambda(t,\tau)\;dt
+\ksi(f(t,t))\,\lambda(t,\tau){\cdot}dW(t),\quad t\in[0,\tau],\\
f(0,\tau)=&\,f_{_0}(\tau)
\end{split}
\end{equation}
for $\tau\in[0,\taumax]$, where
\begin{equation}\label{eq:1.2b}
{\widetilde{\lambda}}(t,\tau) \equiv \lambda(t,\tau)
\cdot\int_t^\tau\lambda(t,z)dz,\quad\forall\,t\in [0,\tau],
\quad\forall\,\tau\in[0,\taumax].
\end{equation}
Here the notation $a\cdot b$ denotes the standard inner product in
$\rset^{\ssy J}$, i.e. $a{\cdot}b\equiv\sum_{j=1}^{\ssy
J}a_j\,b_j$.
In many models used in practice, the function $\lambda$ has the
form $\lambda(t,\tau) = \lambda_{_0}(\tau-t)$, and then
$\tlambda(t,\tau)={\widetilde \lambda}_{_0}(\tau-t)$ with
\begin{equation*}
{\widetilde \lambda}_{_0}(\tau-t)\equiv\lambda_{_0}(\tau-t)\cdot
\int_0^{\tau-t}\lambda_{_0}(z)\;dz.
\end{equation*}
Observe that to solve for $f$ it is enough to have
$\lambda_{_0}:\rset^+\to\rset$.
However, in this work the usual domain of definition ${\mathcal D}$
of $\lambda$ and $f$, extends to the set $[0,\tmax]\times[0,\taumax]$,
leaving $f|_{\ssy {\mathcal D}}$ unchanged.
The extension of ${\mathcal D}$ helps to develop a posteriori
approximations for the time and maturity discretization errors,
depending on a linear backward problem (cf.
Theorem~\ref{thm:3.1}).
%
%
%
%
%
%
\par
A typical contract to price is a call option, with exercise time
$\tmax$ and  strike price $K$, on a zero coupon bond. Its price
can be written in terms of the forward rate as
\begin{equation*}
\expected\left[\,e^{-\int_0^{\tmax}f(s,s)\;ds}\,\,\, \max\left\{
e^{-\int_{\tmax}^{\taumax} f(\tmax,\tau)d\tau}-K,0\right\}\right].
\end{equation*}
Another basic contract  is a continuous cap, with price
\begin{equation*}
\expected\left[\,\int_0^{\tmax}e^{-\int_0^{t}f(s,s)\;ds}\,\,\,
\left(f(t,t)-r_c\right)^{+}\;dt\,\right]
\end{equation*}
where $r_c$ is a given value associated with the contract.
With this motivation, and bearing in mind other possible
contracts, we consider the approximation of the quantity
\begin{equation}\label{eq:1.3}
\expected\left[{\mathcal F}(f)\right]
\end{equation}
where the functional ${\mathcal F}(f)$ is given by
\begin{equation*}
{\mathcal F}(f)\equiv\,F\left(\int_0^{\tmax}f(s,s)\;ds\right)\,
G\left(\int^{\taumax}_{\tau_a}\Psi(f(\tmax,\tau))
\;d\tau\right)
+\int_0^{\tmax}F\left(\int_0^sf(s',s')\;ds'\right)
\,\mcJ(f(s,s))\;ds
\end{equation*}
with $\tau_a$ being a given positive number such that
$0<\tmax\leq\tau_a<\taumax$. Obviously, ${\mathcal F}(f)$ is written
equivalently as
\begin{equation}\label{eq:1.3a}
{\mathcal F}(f)\equiv\,F\left(\,Y(\tmax)\,\right)
\,G\left(\Lambda(\Psi(f(\tmax,\cdot)))\right)
+Z(\tmax),
\end{equation}
 where
\begin{equation}\label{eq:1.3b}
\begin{gathered}
Y(t)\equiv\int_0^{t}f(s,s)\;ds,\quad
Z(t)\equiv\int_0^{t}F(Y(s))\,\mcJ(f(s,s))\;ds,\\
\Lambda(w)\equiv\int^{\taumax}_{\tau_a}
w(\tau)\;d\tau,\quad\forall\,w\in L^1(\tau_a,\taumax).
\end{gathered}
\end{equation}
The functions $F:\rset\rightarrow\rset$,
$G:\rset\rightarrow\rset$, $\Psi:\rset\rightarrow\rset$,
$\mcJ:\rset\rightarrow\rset$, and their derivatives up to a
sufficiently large order $m_{\star}$ are assumed to have a polynomial
growth. We say that a function $S:\rset\to\rset$ has a {\it
polynomial growth} if there exist positive constants $k'$ and $C'$
such that: $|S(x)|\leq C\,'(1+|x|^{k'})$ for all $x\in \rset$.
%
%
%
%
%
\par
Let us consider the system of differential equations \eqref{eq:1.2a}-\eqref{eq:1.2b}
describing the dynamics for the forward rate $f$ along with that for $Y(t)$ and $Z(t)$,  i.e.,
\begin{equation}\label{eq:1.4a}
\begin{split}
df(t,\tau)=&\,\ksi^2(f(t,t))\,\tlambda(t,\tau)\,dt
+\ksi(f(t,t))\,\lambda(t,\tau){\cdot}dW(t),\\
dY(t)=&\,f(t,t)\;dt,\\
dZ(t)=&\,F(Y(t))\,\mcJ(f(t,t))\;dt,\\
\end{split}
\end{equation}
for $t\in[0,\tmax]$ and $\tau\in[0\taumax]$, with the initial conditions
\begin{equation}\label{eq:1.4b}
f(0,\tau)= f_{_0}(\tau),\quad Y(0)= 0,\quad Z(0)=0
\end{equation}
for $\tau\in[0\taumax]$.
\par
A approximation error for a typical discretization of the problem above will consists of a $t-$discretization error and a $\tau-$discretization error
coming from the discretization of the initial condition $f_{_0}$.  Due to the special structure of
\eqref{eq:1.4a}-\eqref{eq:1.4b}, the initial error can be avoided and practically included in the $t-$discretization error
by introducing the anzatz
\begin{equation*}
g(t,\tau)=f(t,\tau)-f_0(\tau),
\end{equation*}
which implies $f(t,t)=g(t,t)+f_0(t)$. Thus,
\eqref{eq:1.4a}-\eqref{eq:1.4b} is formulated as follows:  find
$g=g(t,\tau):[0,\tmax]\times[0,\taumax]\rightarrow{\mathbb R}$
such that
\begin{equation}\label{eq:1.5a}
\begin{split}
dg(t,\tau)=&\,\ksi^2(g(t,t)+f_{_0}(t))\,\tlambda(t,\tau)\;dt
+\ksi(g(t,t)+f_{_0}(t))\,\lambda(t,\tau){\cdot}dW(t),\quad\forall\,t\in[0,\tmax],\\
dY(t)=&\,(g(t,t)+f_{_0}(t))\;dt,\\
dZ(t)=&\,F(Y(t))\,\mcJ(g(t,t)+f_{_0}(t))\;dt,\\
\end{split}
\end{equation}
for $t\in[0,\tmax]$, with  homogeneous initial conditions
\begin{equation}\label{eq:1.5b}
g(0,\tau)= 0,\quad Y(0)= 0,\quad Z(0)= 0
\end{equation}
for all $\tau\in[0,\taumax]$. Thus, the quantity we want to approximate
takes the form
\begin{equation}\label{1:6a}
\expected\left[{\mathcal G}(g)\right]
\end{equation}
where
\begin{equation}\label{1:6b}
{\mathcal G}(g):={\mathcal F}(g+f_{_0}).
\end{equation}
%
%
%
%
%
%
In the numerical methods, we describe later,  the approximations to $Y$ and
$Z$ will be always considered to be respectively the last two components
of the approximate solution vector.
%
%
\subsection{Overview}
Let us give an overiview of the  is organized as follows. In Section \ref{sec:2} first
we present  two  Monte Carlo Euler methods for the HJM model
\eqref{eq:1.5a}-\eqref{eq:1.5b}, namely, a stochastic finite difference
method, the Euler Finite Difference method (EFD), and
a more accurate stochastic finite element method, the Euler
Finite Element method (EFE); then, we combine a numerical quadrature rule
and the outcome of the (EFD) or the (EFE) methods to construct a numerical
approximation of the functional  $\expected\left[{\mathcal G}(g)\right]$.
%
%
%
In Section~\ref{Section_Strong} we provide a stong convergence analysis
for the (EFD) and the (EFE) methods.
Section \ref{sec:3} states and proves weak error estimates
for the (EFD) method, giving explicit formulas for efficient
computation of the discrete duals.
Finally, Section \ref{sec:4} presents results from numerical experiments.
%
%


\section{Monte Carlo Euler Methods}\label{sec:2}
In  this section first we introduce two time and maturity time
discretizations of \eqref{eq:1.5a}-\eqref{eq:1.5b}: the Euler-Finite
Difference (EFD) method and the  Euler-Finite Element (EFE)
method. Then, we use the (EFD) or the (EFE) approximations along with a 
quadrature rule to construct approximations of the quantity of interest
$\expected\left[\mcG(g)\right]$ defined in \eqref{1:6a}.
\subsection{Time and  maturity time discretization}
Given extreme points $0<\tmax\leq\tau_a<\taumax$ introduced in
Section~\ref{Section1}, let $N$ and $L$ denote the number of subintervals
on $[0,\tmax]$ and $[0,\taumax]$, respectively. Then, consider
partitions
\begin{equation*}
0=t_0<\cdots<t_{\ssy N}=\tmax\quad\text{\rm and}\quad
0=\tau_0<\cdots<\tau_{\ssy L}=\taumax
\end{equation*}
of the $t$-interval $[0,\tmax]$ and of the $\tau$-interval
$[0,\taumax]$, respectively. For technical reasons, these
partitions are assumed to satisfy the following condition: every
$\tau$-node in the interval $[0,\tmax]$ is also a $t$-node, i.e.
\begin{equation}\label{eq:A.1}
\text{\rm there exists an one-to-one index
map}\,\rhostar,\,\text{\rm such that},\ken
\tau_{\ell}=t_{\rhostar(\ell)} \ken\text{\rm
for}\ken\tau_\ell\le\tmax.
\end{equation}
In addition, assume that
\begin{equation}\label{eq:A.2}
\text{\rm there exists an index}\ken\lstar
\ken\text{\rm such that}\ken\tmax=\tau_{\lstar}
\end{equation}
and
\begin{equation}\label{eq:A.3}
\text{\rm there exists an index}\ken\ell_a\ken\text{\rm such
that}\ken\tau_a=\tau_{\ell_a}.
\end{equation}
Also, define the auxiliary index function, $\ell_n$, by
\begin{equation}\label{eq:A.4}
\ell_n\equiv \max\left\{\ell\in{\zset}:\quad 0\le \ell \le L\quad
\text{\rm such that}\quad\tau_{\ell}\leq t_n\right\}
\end{equation}
introduce the notation
\begin{equation*}
\Delta t_n \equiv t_{n+1}-t_n,\ken\Delta W_n\equiv
W(t_{n+1})-W(t_n)\quad\text{\rm for}\quad n=0,\dots,N-1,
\end{equation*}
\begin{equation*}
\Delta\tau_\ell\equiv\tau_{\ell+1}-\tau_\ell\quad\text{\rm
for}\quad \ell=0,\dots,L-1,
\end{equation*}
and set $\Delta{t}\equiv\max_{0\leq{n}\leq{\ssy N-1}}\Delta t_n$
and $\Delta\tau\equiv\max_{0\leq{\ell}\leq{\ssy
L-1}}\Delta\tau_{\ell}$.
Finally, introduce the  space of piecewise constant and right
continuous functions on a $\tau$-partition,
$(\tau_{\ell})_{\ell=0}^{\ssy L}$, of the interval $[0,\taumax]$,
by
\begin{equation*}
\xwros\equiv\left\{\chi\in L^{\infty}(0,\taumax): \text{\rm there
are constants}\ken(c_{\ell})_{\ell=0}^{\ssy L-1} \ken\text{\rm
such that}\ken \chi|_{[\tau_{\ell},\tau_{\ell+1})}=c_{\ell},
\quad\ell=0,\dots,L-1\right\}.
\end{equation*}
Define the standard $L^2$--projection
$\prj:L^2(0,\taumax)\velos\xwros$ by
\begin{equation*}
\int_0^{\taumax}\prj{v}\,\chi\;d\tau
=\int_0^{\taumax}v\,\chi\;d\tau, \quad\forall\,\chi\in\xwros,
\quad\forall\,v\in L^2(0,\taumax),
\end{equation*}
which satisfies
\begin{equation*}
\prj{v}\left|_{[\tau_{\ell},\tau_{\ell+1})}\right.
=\tfrac{1}{\Delta\tau_\ell}\int_{\tau_\ell}^{\tau_{\ell+1}}
v(\tau) \ken d\tau,\quad \ell=0,\dots,L-1, \quad\forall\,v\in
L^2(0,\taumax).
\end{equation*}
For $\chi\in\xwros$ and $\ell=0,\dots,L-1$, denote by
$\chi_{\ell}$ the constant value of $\chi$ in
$[\tau_\ell,\tau_{\ell+1})$. When considering a function, $w =
w(t,\tau)$, depending on two variables, the $L^2$ projection is
always with respect to $\tau$, i.e. for $\ell=0,\dots,L-1$ and
$\tau \in [\tau_\ell,\tau_{\ell+1})$, we have
$\prj{w}(t;\tau)\equiv\prj(w(t,\cdot))\left|_{[\tau_{\ell},\tau_{\ell+1})}\right.
=\tfrac{1}{\Delta\tau_\ell}\int_{\tau_\ell}^{\tau_{\ell+1}}w(t,s)
\;ds$.
\subsection{The Euler-Finite Difference (EFD) method}
For each time level the (EFD) method approximates $g(t_n,.)$ by a
piecewise constant function, $\barrg(t_n,.)\in \xwros$. In
particular, it finds the approximate values
$\barrg_{n,\ell}\approx g(t_n,\tau_\ell)$ for $\ell=0,\dots,L-1$,
$\barrg_{n,{\ssy L}}\approx Y(t_n)$, $\barrg_{n,{\ssy L+1}}\approx
Z(t_n)$ by setting first
\begin{equation}\label{eq:2.1a}
\barrg_{0,\ell}\equiv 0,\quad \ell=0,\dots,L+1,
\end{equation}
and, then recursively, 
for $n=0,\dots,N-1$, define
\begin{equation}\label{eq:2.1b}
\begin{split}
\barrg_{n+1,\ell}=&\,\barrg_{n,\ell}+\Delta
t_n\,\ksi^2(\barrg_{n,\ell_n}+f_{_0}(t_n))
\,\tlambda(t_n,\tau_{\ell})\\
&\hskip0.8truecm +\ksi\left(\barrg_{n,\ell_n}+f_{_0}(t_n)\right)
\,\lambda(t_n,\tau_{\ell}){\cdot}\Delta W_n,
\quad\ell=0,\dots,L-1,\\
\barrg_{n+1,{\ssy L}}=&\,\barrg_{n,{\ssy L}}
+\Delta t_n\,\left(\barrg_{n,\ell_n}+f_{_0}(t_n)\right),\\
\barrg_{n+1,{\ssy L+1}}=&\,\barrg_{n,{\ssy L+1}}+\Delta
t_n\,\,\,F(\barrg_{n,{\ssy L}})
\,\,\,U\left(\barrg_{n,\ell_n}+f_{_0}(t_n)\right)\\
\end{split}
\end{equation}
where
%
%
the index $\ell_n$ has been defined in \eqref{eq:A.4}.
%
%
%
\subsection{ The Euler-Finite Element (EFE) method}
The (EFE) method also approximates the $\tau$-function
$g(t_n,\cdot)$, by a piecewise constant function
$\barrg(t_n,\cdot)\in\xwros$, but is based in a variational
formulation of \eqref{eq:1.5a}-\eqref{eq:1.5b} with $\xwros$ being
the space of trial and test functions.
In particular, the (EFE) is defined by the initial datum
\begin{equation}\label{eq:2.2a}
\barrg_{0,\ell}\equiv  0,\quad \ell=0,\dots,L+1,
\end{equation}
and, for $n=0,\dots,N-1$, the recursion
\begin{equation}\label{eq:2.2b}
\begin{split}
\barrg_{n+1,\ell}=&\,\barrg_{n,\ell}+{\Delta
t}_n\,\ksi^2\left(\barrg_{n,\ell_n}+f_{_0}(t_n)\right)
\,\prj\tlambda(t_n;\tau_\ell)\\
&\hskip0.8truecm
+\ksi\left(\barrg_{n,\ell_n}+f_{_0}(t_n)\right)\,\prj\lambda(t_n;\tau_\ell)
{\cdot}\Delta W_n,\quad\ell=0,\dots,L-1,\\
\barrg_{n+1,{\ssy L}}=&\,\barrg_{n,{\ssy L}}+\Delta t_n\,\left(\barrg_{n,\ell_n}+f_{_0}(t_n)\right),\\
\barrg_{n+1,{\ssy L+1}}=&\,\barrg_{n,{\ssy L+1}}
+\Delta t_n\,F(\barrg_{n,{\ssy L}})\,U\left(\barrg_{n,\ell_n}+f_{_0}(t_n)\right)\\
\end{split}
\end{equation}
where the index $\ell_n$ has been defined in \eqref{eq:A.4}.
%
%
%
%
%
\subsection{Approximation of the quantity of interest $\expected[\mcG(g)]$}
The numerical approximation of $\mcG(g)$ defined in \eqref{1:6a}
involves both an approximation of the processes $g$, $Y$, $Z$, by
computable quantities, and an approximation of the $\tau$-integral
in \eqref{eq:1.3b}.
\par
To construct an approximation of $\Lambda(\Psi(g(\tmax,\cdot)+f_{_0}(\cdot)))$ we
apply a composite quadrature formula, over the partition of
$[0,\taumax]$, based on a quadrature rule
$Q:C[0,1]\rightarrow{\mathbb R}$ with $N_{\ssy Q}$ nodes
$s_{\ssy Q}=(s_{{\ssy Q},i})_{i=1}^{\ssy N_{Q}}$ and weights
$w_{\ssy Q}=(w_{{\ssy Q},i})_{i=1}^{\ssy N_{Q}}$,
i.e., for $v\in C([0,1];{\mathbb R})$ the quantity $Q(v)=\sum_{i=1}^{\ssy
N_{Q}}w_{{\ssy Q},i}\,v(s_{{\ssy Q},i})$ approximates the integral $\int_0^1v(x)\;dx$.
Also, we assume that the quadrature rule $Q$ is of order $p_{\ssy
Q}$, i.e., it is exact for polynomials of order less or equal to
$p_{\ssy Q}-1$.
For example, the Simpson rule has $N_{\ssy Q}=3$,
$s_{\ssy Q}=(0,\half,1)$ and $w_{\ssy Q}=(\frac{1}{6},\frac{2}{3},\frac{1}{6})$,
with $p_{\ssy Q}=4$. Another example is the Gaussian quadrature with
$N_{\ssy Q}=2$, 
$s_{\ssy Q}=(\half-\frac{1}{2\sqrt{3}},\half+\frac{1}{2\sqrt{3}})$,
$w_{\ssy Q}=(\half,\half)$ and $p_{\ssy Q}=4$. We note that it is well
known from the mathematical analysis of numerical quadrature that in general we have
$p_{\ssy Q}\leq\,2\,N_{\ssy Q}$, and the maximum value $p_{\ssy Q}=2N_{\ssy Q}$
is achieved only by the Gaussian quadrature.
\par
Thus, for a fixed realization of $\barrg$ obtained by the (EFD) or the (EFE) method, first
we approximate $\Lambda_{\ssy\Psi}(g):=\Lambda(\Psi(g(\tmax,\cdot)+f_{_0}(\cdot)))$ by
$\Lambda_{\ssy\Psi}(\bbg)=\Lambda(\Psi(\barrg(\tmax,\cdot)+f_{_0}(\cdot)))$ and then we apply the composite
quadrature formula to construct an approximation 
${\overline\Lambda}_{\ssy\Psi,Q}(\bbg)$ of $\Lambda_{\ssy\Psi}(\bbg)$ as follows
\begin{equation}\label{eq:2.3a}
\begin{split}
{\overline\Lambda}_{\ssy\Psi,Q}(\bbg)=&\,\sum_{\ell=\ell_a}^{\ssy
L-1}\Delta\tau_{\ell}\,Q\left(\Psi\left(\,
\barrg(\tmax,\tau_{\ell}+\cdot\,\Delta\tau_{\ell})
+f_{_0}(\tau_{\ell}+\cdot\,\Delta\tau_{\ell})\,\right)\right)\\
=&\sum_{\ell=\ell_a}^{\ssy L-1}\Delta\tau_{\ell}\,\left[
\,\sum_{i=1}^{\ssy N_{\ssy Q}}w_{{\ssy Q},i}\,\Psi\left(\barrg_{{\ssy
N},\ell} +f_{_0}(\tau_{\ell}+s_{{\ssy Q},i}\,\Delta\tau_{\ell})\right)
\,\right].\\
\end{split}
\end{equation}
%
%
%
%
%
%
Note that $\barrg(\tmax,\cdot)$ is piecewise constant over the partition
of $[0,\taumax]$  and numerical quadrature error in \eqref{eq:2.3a}
is caused only from the presence of the initial datum $f_{_0}$.
In particular, if the initial datum for the term structure, $f_{_0}$, is a
piecewise constant function on the maturity time partition, then there
is no quadrature error.
%
%
Finally, an approximation ${\brr{\mcG}}(\barrg)$ of
${\mcG}(g)$ is computed by
\begin{equation}\label{eq:2.4a}
{\brr{\mcG}}(\barrg)\equiv F\bigl(\barrg_{{\ssy N},{\ssy
L}}\bigr)\,G\left({\overline\Lambda}_{\ssy \Psi,Q}(\bbg)\right)+\barrg_{{\ssy
N},{\ssy L+1}}.
\end{equation}
\par
The Monte Carlo method, \cite{MonteC}, approximates the
expectation of a given random variable $X$ by a sample average of
$M$ independent realizations of $X$, i.e. $\expected[X]\approx
\mcA(M;X)\equiv\tfrac{1}{M}\sum_{j=1}^{\ssy M} X(\omega_j)$.
In particular, here we approximate $\expected[{\mathcal G}(g)]$ by a
sample average of ${\brr{\mcG}}(\barrg)$,
\begin{equation}\label{eq:2.4b}
\mcA\left(M;{\brr{\mcG}}(\barrg)\right)\equiv
\tfrac{1}{M}\sum_{j=1}^{\ssy M} \left[\,F\left(\barrg_{{\ssy
N},{\ssy L}}(\omega_j)\right)\,
G\left(\,{\overline\Lambda}_{\ssy\Psi,Q}(\bbg(\omega_j))\,\right)
+\barrg_{{\ssy N},{\ssy L+1}}(\omega_j)\,\right].
\end{equation}
\par
Therefore, the exact computational weak error
\begin{equation}\label{Comp_Error}
{\mcE}_c\equiv \expected[{\mathcal
G}(g)]-\mcA\left(M;{\brr{\mcG}}(\barrg)\right)
\end{equation}
naturally separates into three error contributions as
follows:
\begin{equation}\label{eq:2.5}
{\mcE}_c=E_{\ssy D} +E_{\ssy Q} + E_{\ssy S}
\end{equation}
%
%
with
\begin{equation}\label{eq:2.5a}
\begin{gathered}
E_{\ssy D}\equiv\,\expected\left[\,\mcG(g)\,\right]
-\expected\left[\,{\mcG}(\barrg)\,\right],\quad
E_{\ssy Q} \equiv\,\expected\left[\,\mcG(\barrg)\,\right]
-\expected\left[\,{\overline{\mcG}}(\barrg)\,\right],\\
E_{\ssy
S}\equiv\,\expected\left[\,{\overline{\mcG}}(\barrg)\,\right]-{\mcA}
\left(M;{\overline{\mcG}}(\barrg)\right)\\
\end{gathered}
\end{equation}
where $E_{\ssy D}$ is the error contribution from $t$- and $\tau$-
discretization, $E_{\ssy Q}$ is the quadrature error in
\eqref{eq:2.3a}, and $E_{\ssy S}$ is the statistical error.
%
%
%
%
%
\section{Strong Convergence}\label{Section_Strong}
To carry out an error analysis for the numerical methods proposed
in Section~\ref{sec:2}, we assume that there exists nonnegative
constants $C_{\xi,1}$ and $C_{\xi,2}$ such that
\begin{equation}\label{KSI_1}
|\xi^2(x)|\leq\,C_{\xi,1}\,(1+|x|)\quad\forall\,x\in{\mathbb R},
\end{equation}
and
\begin{equation}\label{KSI_2}
|\ksi^2(x)-\ksi^2(z)|+|\ksi(x)-\ksi(z)|\leq\,C_{\xi,2}\,|x-z|,
\quad\forall\,x,\,z\in{\mathbb R}.
\end{equation}
\subsection{Bounds for moments}
In Lemmas~\ref{AUXLM_2} and \ref{AUXLM_14}, we show, respectively,
boundness for the moments of the $\tau-$derivatives of the
solution $g$ to the problem \eqref{eq:1.5a}--\eqref{eq:1.5b}, and
for the functional value ${\mathcal G}(g)$.
\begin{lem}\label{AUXLM_2}
Let $D_{\star}\equiv[0,t_{\max}]\times[0,\tau_{\max}]$, $g$ be the
solution of \eqref{eq:1.5a}--\eqref{eq:1.5b} and $\nu\in{\mathbb
N}_0$. Also, we assume that the derivatives
$\partial_{\tau}^{\ell}{\widetilde\lambda}$ and
$(\partial_{\tau}^{\ell}\lambda_j)_{j=1}^{\ssy J}$ are well
defined and continuous on $D_{\star}$, for $\ell=0,\dots,\nu$.
Then, for $\ell=0,\dots,\nu$ and $\kappa\in{\mathbb N}$, there
exists a positive constant $C^{\ssy M}_{\kappa,\ell}$, depending on
$\kappa$, $\ell$, $(\partial_{\tau}^{\ell}\lambda_j)_{j=1}^{\ssy
J}$, $\partial_{\tau}^{\ell}{\widetilde\lambda}$, $f_{_0}$,
$C_{\xi,1}$, $\tau_{\max}$ and $t_{\max}$, such that
\begin{equation}\label{KSI_MaximusM_1}
\max_{\ssy (t,\tau)\in D_{\star}}\expected\left[\,
\left|\partial_{\tau}^{\ell}g(t,\tau)\right|^{2\kappa}
\,\right]\leq\,C^{\ssy M}_{\kappa,\ell},
\end{equation}
where $C_{\xi,1}$ is the constant in \eqref{KSI_1}.
\end{lem}
%
%
%
%
%
%
%
%
%
%
%
%
%
\begin{proof}
Let $\kappa\in{\mathbb N}$, $\ell\in\{0,\dots,\nu\}$ and
$(t,\tau)\in D_{\star}$. Also, in order to simplify the notation,
we set $t_{\star}:=\tmax$ and $\tau_{\star}:=\taumax$.
Our first step is to use \eqref{eq:1.5a} to get
\begin{equation}\label{KSI_M3}
\expected\left[\,\left|\partial_{\tau}^{\ell}g(t,\tau)\right|^{2\kappa}\,\right]
\leq\,(J+1)^{2\kappa-1}\,\left[\,T_{1,\kappa}^{\ell}(t,\tau)
+T^{\ell}_{2,\kappa}(t,\tau)\,\right],
\end{equation}
where
\begin{equation*}
\begin{split}
T_{1,\kappa}^{\ell}(t,\tau)\equiv&\,\expected\left[\,\left(\int_0^t
\partial_{\tau}^{\ell}{\widetilde\lambda}(s,\tau)
\,\xi^2(g(s,s)+f_{_0}(s))\;ds\right)^{2\kappa}\right],\\
T_{2,\kappa}^{\ell}(t,\tau)\equiv&\,\sum_{j=1}^{\ssy
J}\,\expected\left[\,\left(\int_0^t\partial_{\tau}^{\ell}\lambda_j(s,\tau)
\,\ksi(g(s,s)+f_{_0}(s))\;dW^j(s)\right)^{2\kappa}\right].\\
\end{split}
\end{equation*}
Using \eqref{KSI_1} and applying the H{\"o}lder inequality we have
\begin{equation}\label{KSI_M4}
\begin{split}
T_{1,\kappa}^{\ell}(t,\tau)\leq&\,(C_{\xi,1})^{2\kappa}\,
\expected\left[\,\left(\int_0^t|\partial_{\tau}^{\ell}{\widetilde\lambda}(s,\tau)|
\,\left(\,1+|f_{_0}(s)|+|g(s,s)|\,\right)\;ds\right)^{2\kappa}\right]\\
\leq&\,2^{2\kappa-1}\,(C_{\xi,1})^{2\kappa}\,
\expected\Bigg[\,\left(\int_0^t|\partial_{\tau}^{\ell}{\widetilde\lambda}(s,\tau)|
\,(1+|f_{_0}(s)|)
\,ds\right)^{2\kappa}\\
&\hskip3.5truecm
+\left(\int_0^t|\partial_{\tau}^{\ell}{\widetilde\lambda}(s,\tau)|\,
\,|g(s,s)|\;ds\right)^{2\kappa}\Bigg]\\
\leq&\,C_{1}^{\kappa,\ell}+C_{2}^{\kappa,\ell}\,\int_0^t\expected
\left[\,(g(s,s))^{2\kappa}\,\right]\;ds,\\
\end{split}
\end{equation}
where $C_{1}^{\kappa,\ell}\equiv
2^{2\kappa-1}\,(C_{\xi,1})^{2\kappa}\,
\max\limits_{\tau\in[0,\tau_{\star}]}\left(\int_0^{t_{\star}}
|\partial_{\tau}^{\ell}{\widetilde\lambda}(s,\tau)|
\,(1+|f_{_0}(s)|)\;ds\right)^{2\kappa}$ and
\begin{equation*}
C_{2}^{\kappa,\ell}\equiv 2^{2\kappa-1}\,(C_{\xi,1})^{2\kappa}
\,\max\limits_{\tau\in[0,\tau_{\star}]}\left(\int\nolimits_0^{t_{\star}}|\partial_{\tau}^{\ell}
{\widetilde\lambda}(s,\tau)|^{\frac{2\kappa}{2\kappa-1}}\;ds\right)^{2\kappa-1}.
\end{equation*}
Next, using the properties of the It$\hat{\text{\rm o}}$ integral
and \eqref{KSI_1}, we obtain
\begin{equation}\label{KSI_M5}
\begin{split}
T_{2,\kappa}^{\ell}(t,\tau)\leq&\,(2\kappa-1)!!\,\sum_{j=1}^{\ssy
J}\left(\int_0^t(\partial_{\tau}^{\ell}\lambda_j(s,\tau))^2
\,\expected\left[\,\ksi^2(g(s,s)+f_{_0}(s))\,\right]\;ds\right)^{\kappa}\\
\leq&\,(2\kappa-1)!!\,(C_{\xi,1})^{\kappa}\,\sum_{j=1}^{\ssy
J}\left(\int_0^t(\partial_{\tau}^{\ell}\lambda_j(s,\tau))^2
\,\left(1+|f_{_0}(s)|+\expected\left[\,|g(s,s)|\,\right]
\,\right)\;ds\right)^{\kappa}\\
\leq&\,(2\kappa-1)!!\,(C_{\xi,1})^{\kappa}\,\sum_{j=1}^{\ssy
J}\left(\int_0^t(\partial_{\tau}^{\ell}\lambda_j(s,\tau))^2
\,\left(2+|f_{_0}(s)|+\expected\left[\,|g(s,s)|^2\,\right]
\right)\;ds\right)^{\kappa}\\
\leq&\,C_{4}^{\kappa,\ell}+C_{3}^{\kappa,\ell}\,
\left(\int_0^t\expected\left[\,|g(s,s)|^2
\,\right]\;ds\right)^{\kappa},\\
\end{split}
\end{equation}
where
$C_{3}^{\kappa,\ell}\equiv\,(2\kappa-1)!!\,2^{\kappa-1}
\,(C_{\xi,1})^{\kappa}\,\left(\sum_{j=1}^{\ssy
J}\max\limits_{\ssy
D_{\star}}|\partial_{\ssy\tau}^{\ell}\lambda_j|^{2\kappa}\right)$
and
$C_{4}^{\kappa,\ell}\equiv\,C_{3}^{\kappa,\ell}\,\left(\int_0^{\tmax}
\,\left(2+|f_{_0}(s)|\right)\;ds\right)^{\kappa}$.
Now, combine \eqref{KSI_M3}, \eqref{KSI_M4} and \eqref{KSI_M5}, to
arrive at
\begin{equation}\label{KSI_M6}
\expected\left[\,\left(\partial_{\tau}^{\ell}g(t,\tau)\right)^{2\kappa}
\,\right]\leq\,C_{\ssy I}^{\kappa,\ell}+C_{\ssy
I\!I}^{\kappa,\ell}\left(\int_0^t\expected
\left[\,|g(s,s)|^{2}\,\right]\;ds\right)^{\kappa}
+C_{\ssy I\!I\!I}^{\kappa,\ell}\int_0^t\expected
\left[\,|g(s,s)|^{2\kappa}\,\right]\;ds,\\
\end{equation}
where $C_{\ssy I}^{\kappa,\ell}=(J+1)^{2\kappa-1}
\,(\,C_{1}^{\kappa,\ell}+C_{4}^{\kappa,\ell}\,)$, $C_{\ssy
I\!I}^{\kappa\,\ell}=(J+1)^{2\kappa-1}\,C_{2}^{\kappa,\ell}$ and
$C_{\ssy
I\!I\!I}^{\kappa,\ell}=(J+1)^{2\kappa-1}\,C_{3}^{\kappa,\ell}$.
\par
Consider the case $\kappa=1$ and $\ell=0$, and set $\tau=t$ in
\eqref{KSI_M6}, to obtain
\begin{equation*}
\expected\left[\,|g(t,t)|^{2}\,\right]\leq\,C_{\ssy
I}^{1,0}+\left(\,C_{\ssy I\!I}^{1,0}+C_{\ssy
I\!I\!I}^{1,0}\right)\, \int_0^{t}\expected\left[\,|g(s,s)|^{2}
\,\right]\;ds,\quad\forall \,t\in[0,t_{\star}],
\end{equation*}
which, after the application of the Gr{\"o}nwall lemma, yields
\begin{equation}\label{KSI_M7}
\expected\left[\,|g(t,t)|^2\,\right]\leq\,C_{\ssy
I}^{1,0}\,e^{(C_{\ssy I\!I}^{1,0}+C_{\ssy I\!I\!I}^{1,0})\,t},
\quad\forall\,t\in[0,t_{\star}].
\end{equation}
Now, combine \eqref{KSI_M7} and \eqref{KSI_M6} (with $\kappa=1$),
to get
\begin{equation}\label{KSI_M505}
\expected\left[\,\left|\partial_{\tau}^{\ell}g(t,\tau)\right|^{2}
\,\right]\leq\,C_{\ssy I}^{1,\ell}+\tfrac{(C_{\ssy
I\!I}^{1,\ell}+C_{\ssy I\!I\!I}^{1,\ell})\,C_{\ssy
I}^{1,0}}{C_{\ssy I\!I}^{1,0}+C_{\ssy
I\!I\!I}^{1,0}}\left[\,e^{(\,C_{\ssy I\!I}^{1,0}+C_{\ssy
I\!I\!I}^{1,0})\,t}-1\,\right],\quad\forall\,(t,\tau)\in
D_{\star},
\end{equation}
for $\ell=0,\dots,\nu$, which establishes \eqref{KSI_MaximusM_1}
for $\kappa=1$.
\par
Now, consider the case $\kappa\ge2$. Then, use \eqref{KSI_M7} and
\eqref{KSI_M6}, to obtain
\begin{equation}\label{KSI_M506}
\expected\left[\,\left|\partial_{\tau}^{\ell}g(t,\tau)\right|^{2\kappa}
\,\right]\leq\,C_{\ssy I\!V}^{\kappa,\ell}+C_{\ssy
I\!I\!I}^{\kappa,\ell}\int_0^t\expected
\left[\,|g(s,s)|^{2\kappa}\,\right]\;ds,\quad\forall\,(t,\tau)\in
D_{\star},\ \ \ell=0,\dots,\nu,
\end{equation}
where $C_{\ssy I\!V}^{\kappa,\ell}=C_{\ssy
I}^{\kappa,\ell}+C_{\ssy I\!I}^{\kappa,\ell}\,\left[\tfrac{C_{\ssy
I}^{1,0}}{C_{\ssy I\!I}^{1,0}+C_{\ssy I\!I\!I}^{1,0}}\,
\left(\,e^{(\,C_{\ssy I\!I}^{1,0}+C_{\ssy
I\!I\!I}^{1,0})\,\tmax}-1\right)\right]^{\kappa}$. Take $\ell=0$
and set $\tau=t$ in \eqref{KSI_M506}, to obtain
\begin{equation*}
\expected\left[\,|g(t,t)|^{2\kappa}\,\right]\leq\,C_{\ssy
I\!V}^{\kappa,0}+C_{\ssy I\!I\!I}^{\kappa,0}\,
\int_0^{t}\expected\left[\,|g(s,s)|^{2\kappa}
\,\right]\;ds,\quad\forall \,t\in[0,t_{\star}].
\end{equation*}
Apply again the Gr{\"o}nwall lemma, to conclude that
\begin{equation}\label{KSI_M807}
\expected\left[\,|g(t,t)|^{2\kappa}\,\right]\leq\,C_{\ssy
I\!V}^{\kappa,0}\,e^{C_{\ssy I\!I\!I}^{\kappa,0}\,t},
\quad\forall\,t\in[0,t_{\star}].
\end{equation}
Finally, combine \eqref{KSI_M807} and \eqref{KSI_M506} to have
\begin{equation*}\label{KSI_71}
\expected\left[\,|\partial_{\tau}^{\ell}g(t,\tau)|^{2\kappa}\,\right]\leq\,C_{\ssy
I\!V}^{\kappa,\ell}+\tfrac{C_{\ssy I\!I\!I}^{\kappa,\ell}\,C_{\ssy
I\!V}^{\kappa,0}}{C_{\ssy I\!I\!I}^{\kappa,0}}\,(e^{C_{\ssy
I\!I\!I}^{\kappa,0}\,t}-1),\quad\forall\,(t,\tau)\in
D_{\star},\quad \ell=0,\dots,\nu,
\end{equation*}
which yields the desired bound \eqref{KSI_MaximusM_1} for
$\kappa\ge2$.
\end{proof}
%
%
%
%
%
%
%
%
%
\begin{lem}\label{AUXLM_14}
Let $(g,Y,Z)$ be the solution of the system
\eqref{eq:1.5a}--\eqref{eq:1.5b}. Also, we assume that the functions $F$, $G$, $\Psi$,
$U:{\mathbb R}\rightarrow{\mathbb R}$ have polynomial growth $p_{\ssy F}$, $p_{\ssy G}$,
$p_{\ssy\Psi}$ and $p_{\ssy U}$
with constants $C_{\ssy F}$, $C_{\ssy G}$, $C_{\ssy\Psi}$ and
$C_{\ssy U}$, respectively.
Then, for $\kappa\in{\mathbb N}$, there exists a positive constant
$C_{\kappa}$, depending on $\kappa$ and the data of the problem,
such that
\begin{equation}\label{Moments_of_Functional}
\expected\left[\,\left|F(Y(\tmax))\right|^{2\kappa}
\,\right]+\expected\left[\, \left|G(\Lambda(\Psi(g(\tmax,\cdot)+f_{_0})))\right|^{2\kappa}
\,\right]+\expected\left[\, \left|Z(\tmax)\right|^{2\kappa}
\,\right]\leq\,C_{\kappa}.
\end{equation}
\end{lem}
%
%
%
%
%
%
%
%
%
%
\begin{proof}
Let $\kappa\in{\mathbb N}$. To simplify the notation, we set $\tau_{\star}:=\taumax$,
$t_{\star}:=\tmax$ and $\Upsilon(\tau):=g(\tmax,\tau)+f_{_0}(\tau)=f(\tmax,\tau)$ for $\tau\in[0,\taumax]$. 
Since $F$, $U$ and $G$ have polynomial growth, using the
H{\"o}lder inequality and \eqref{KSI_MaximusM_1} for $\ell=0$, we
obtain
\begin{equation}\label{MFunct2}
\expected\left[\,|F(Y(t))|^{2\kappa}\,\right]\leq\,(C_{\ssy
F})^{2\kappa}\,\,2^{2\kappa-1}
\,\left(1+\expected\left[\,|Y(t)|^{2 \kappa p_{_{F}}}\,\right]
\,\right),\quad\forall\,t\in[0,t_{\star}],
\end{equation}
\begin{equation}\label{MFunct3}
\begin{split}
\expected\left[\,|Y(t)|^{2 m}
\,\right]\leq&\,(2\,t_{\star})^{2 m-1}
\,\int_0^{t}\left(\,\expected\left[|g(s,s)|^{2 m}\,\right]
+|f_{_0}(s)|^{2 m}\,\right)\;ds\\
\leq&\,(2\,t_{\star})^{2 m-1} \,\int_0^{t_{\star}}\left(\,C^{\ssy
M}_{2 m,0}+|f_{_0}(s)|^{2 m}
\right)\;ds,
\quad\forall\,t\in[0,t_{\star}],
\quad\forall\,m\in{\mathbb N},\\
\end{split}
\end{equation}
\begin{equation}\label{MFunct4}
\begin{split}
\expected\left[\,|Z(t_{\star})|^{2\kappa}\,\right]\leq&\,(t_{\star})^{2\kappa-1}\,
\left\{\,\int_0^{t_{\star}}\expected\left[\,|F(Y(s))|^{4\kappa}\,\right]\;ds
+\int_0^{t_{\star}}\expected\left[\,|U(f(s,s))|^{4\kappa}
\,\right]\;ds\,\right\},\\
\end{split}
\end{equation}
\begin{equation}\label{MFunct5}
\begin{split}
\expected\left[\,|U(f(t,t))|^{2 m}\,\right]
\leq&\,(C_{\ssy U})^{2m}\,3^{2m-1}\,\left(\,1+|f_{_0}(t)|^{2m p_{_U}}
+\expected\left[\,|g(t,t)|^{2m p_{_U}}\,\right]\,\right)\\
\leq&\,(C_{\ssy U})^{2m}\,3^{2m-1}
\,\left(\,1+|f_{_0}(t)|^{2m p_{_U}} +C^{\ssy M}_{2m
p_{_U},0}\,\right),\quad\forall\,t\in[0,t_{\star}],
\quad\forall\,m\in{\mathbb N},
\end{split}
\end{equation}
\begin{equation}\label{MFunct6a}
\expected\left[\,|G(\Lambda(\Psi(\Upsilon)))|^{2\kappa}\,\right]
\leq\,(C_{\ssy G})^{2\kappa}\,\left(1+\expected\left[\,|\Lambda(\Psi(\Upsilon))|^{2\kappa p_{_G}}
\,\right]\right),
\end{equation}
\begin{equation}\label{MFunct6b}
\begin{split}
\expected\left[\,|\Lambda(\Psi(\Upsilon))|^{2m}\,\right]
\leq&\,(\tau_{\star}-\tau_a)^{2m-1}
\,\int_{\tau_a}^{\taumax} \expected\left[\,
|\Psi(\Upsilon(\tau))|^{2m}\,\right]\;d\tau,
\quad\forall\,m\in{\mathbb N},
\end{split}
\end{equation}
and
\begin{equation}\label{MFunct7}
\begin{split}
\expected\left[\,|\Psi(\Upsilon(\tau))|^{2 m}\,\right]\leq&\,
(C_{\ssy\Psi})^{2 m}\,3^{2m-1}\,\left(\,1+\expected\left[\,|g(t_{\star},\tau)|^{2 m p_{_{\Psi}}}\,\right]
+|f_{_0}(\tau)|^{2 m p_{_{\Psi}}}\,\right)\\
\leq&\,(C_{\ssy\Psi})^{2m}\,3^{2m-1}\,\left(\,1+|f_{_0}(\tau)|^{2m p_{_{\Psi}}} +C^{\ssy
M}_{2m p_{_{\Psi}},0}\,\right),
\quad\forall\,\tau\in[0,\tau_{\star}],
\quad\forall\,m\in{\mathbb N}.\\
\end{split}
\end{equation}
\par
Thus, we obtain \eqref{Moments_of_Functional} combining the
inequalities \eqref{MFunct2}-\eqref{MFunct7} above.
\end{proof}
%
%
%
%
\par
In Lemma~\ref{AUXLM_3}  below, we show
boundness for the moments of the numerical approximations produced
by the \text{\rm(EFD)} and the \text{\rm (EFM)} method.
%
%
%
%
%
%
\begin{lem}\label{AUXLM_3}
Let ${\mathcal I}:=\{0,\dots,N\}\times\{0,\dots,L-1\}$ and
$(\bbg_{n,\ell})_{(n,\ell)\in{\mathcal I}}$ be the
numerical approximations produced by the \text{\rm(EFD)} or the
\text{\rm (EFM)} method. Then, for $\kappa\in{\mathbb N}$, there exists
a nonnegative constant $C^{\ssy M}_{{\ssy D},\kappa}$, depending on
$\kappa$, $(\lambda_j)_{j=1}^{\ssy J}$, ${\widetilde\lambda}$,
$f_{_0}$, $C_{\xi,1}$, $\tau_{\max}$ and $t_{\max}$, such that
\begin{equation}\label{DiscreteM_0}
\max_{(n,\ell)\in{\mathcal
I}}\expected\left[\,|\bbg_{n,\ell}|^{2\kappa}\,\right]\leq\,C^{\ssy
M}_{{\ssy D},\kappa},
\end{equation}
where $C_{\ksi,1}$ is the constant in \eqref{KSI_1}.
\end{lem}
%
%
%
%
%
%
%
\begin{proof}
Let $D_{\star}\equiv[0,\tmax]\times[0,\taumax]$,
$\kappa\in{\mathbb N}$, $(n,\ell)\in{\mathcal I}$ with
$n\ge1$. Then, from \eqref{eq:2.1b} and \eqref{eq:2.2b}, we
conclude that
\begin{equation}\label{DiscreteM_1}
\bbg_{n,\ell}=\sum_{m=0}^{n-1}\Delta
t_m\,\ksi^2(\bbg_{m,\ell_m}+f_{_0}^m)\,\nu^{m,\ell}
+\sum_{m=0}^{n-1}\sum_{j=1}^{\ssy J}
\ksi\big(\bbg_{m,\ell_m}+f_{_0}^m\big)
\,\mu^{m,\ell}_j\,\Delta{W_m^j},
\end{equation}
where $f_{_0}^m:=f_{_0}(t_m)$,
$\nu^{m,\ell}={\widetilde\lambda}(t_m,\tau_{\ell})$ and
$\mu_j^{m,\ell}=\lambda_j(t_m,\tau_{\ell})$ for the (EFD) method,
and $\nu^{m,\ell}=\Pi{\widetilde\lambda}(t_m;\tau_{\ell})$ and
$\mu_j^{m,\ell}=\Pi\lambda_j(t_m;\tau_{\ell})$ for the (EFE)
method. Thus, we obtain
\begin{equation}\label{DiscreteM_2}
\expected\left[\,|\bbg_{n,\ell}|^{2\kappa}\,\right]
\leq\,(J+1)^{2\kappa-1}\,(\,T_{1,\kappa}^{n,\ell}+T_{2,\kappa}^{n,\ell}\,)
\end{equation}
where
\begin{equation*}
\begin{split}
T_{1,\kappa}^{n,\ell}:=&\,\expected\left[\,\left(\,\sum_{m=0}^{n-1}\Delta
t_m\,\,\nu^{m,\ell}\,\,\ksi^2\big(\bbg_{m,\ell_m}+f_{_0}^m\big)
\right)^{2\kappa}\,\right],\\
T_{2,\kappa}^{n,\ell}:=&\,\sum_{j=1}^{\ssy
J}\expected\left[\left(\,\sum_{m=0}^{n-1}\,\mu^{m,\ell}_j\,
\ksi(\bbg_{m,\ell_m}+f_{_0}^m)
\,\Delta{W}_m^j\right)^{2\kappa}\right].\\
\end{split}
\end{equation*}
Using \eqref{KSI_1} we bound $T_{1,\kappa}^{n,\ell}$ as follows
\begin{equation*}
\begin{split}
T_{1,\kappa}^{n,\ell}\leq&\,(C_{\xi,1})^{2\kappa}\,\expected\left[
\left(\sum_{m=0}^{n-1}\Delta{t_m}\,|\nu^{m,\ell}|
\,(1+|f_{_0}^m|+|\bbg_{m,\ell_m}|)
\right)^{2\kappa}\right]\\
\leq&\,2^{2\kappa-1}\,(C_{\xi,1})^{2\kappa}\,\expected\left[\,\left(
\sum_{m=0}^{n-1}\Delta{t_m}
\,|\nu^{m,\ell}|\,(1+|f_{_0}^m|)\right)^{2\kappa}
+\left(\sum_{m=0}^{n-1}\Delta{t_m}|\nu^{m,\ell}|
\,|\bbg_{m,\ell_m}|\right)^{2\kappa}\,\right]\\
\end{split}
\end{equation*}
which, after applying the H{\"o}lder inequality, yields
\begin{equation}\label{DiscreteM_3}
T_{1,\kappa}^{n,\ell}\leq\,C_{{\ssy D},1,\kappa}\,\left(\,\tmax
+\sum_{m=0}^{n-1}\Delta{t_m}\,\expected\left[
\,\left|\bbg_{m,\ell_m}\right|^{2\kappa}\right]
\,\right).
\end{equation}
where $C_{{\ssy D},1,\kappa}=(2\,C_{\xi,1})^{2\kappa}
\,(\tmax)^{2\kappa-1}
\,\max_{\ssy D_{\star}}[\,|{\widetilde\lambda}|\,(1+|f_{_0}|)
\,]^{2\kappa}$.
Also, using the properties of independent Gaussian random
variables and \eqref{KSI_1}, we obtain
\begin{equation*}
\begin{split}
T_2^{n,\ell}\leq&\,(2\kappa-1)!!\,\sum_{j=1}^{\ssy J}
\left(\,\sum_{m=0}^{n-1}\Delta{t_m}\,(\mu_j^{m,\ell})^2\,
\expected\left[\,\xi^2(\bbg_{m,\ell_m}+f_{_0}^m)
\,\right]\,\right)^{\kappa}\\
\leq&\,(2\kappa-1)!!\,(C_{\xi,1})^{\kappa} \,\max_{\ssy
D_{\star}}|\lambda|^{2\kappa}\,
\left(\,\sum_{m=0}^{n-1}\Delta{t_m}\,\left(\,2+|f_{_0}^m|
+\expected\left[\,|\bbg_{m,\ell_m}|^2
\,\right]\,\right)\,\right)^{\kappa}\\
\leq&\,(2\kappa-1)!!\,(C_{\xi,1})^{\kappa} \,\max_{\ssy
D_{\star}}|\lambda|^{2\kappa}\,\left[
\tmax\,\max_{[0,\tmax]}\left(2+|f_{_0}|\right)
+\sum_{m=0}^{n-1}\Delta{t_m}\,\expected\left[\,|\bbg_{m,\ell_m}|^2
\,\right]\right]^{\kappa}\\
\end{split}
\end{equation*}
which  yields that
\begin{equation}\label{DiscreteM_4}
T_2^{n,\ell}\leq\,C_{{\ssy D},2,\kappa}\,\left[
(\tmax)^{\kappa}
+\left(\,\sum_{m=0}^{n-1}\Delta{t_m}
\,\expected\left[\,|\bbg_{m,\ell_m}|^2
\,\right]\,\right)^{\kappa}\,\right]
\end{equation}
where $C_{{\ssy D},2,\kappa}=2^{\kappa-1}\,(2\kappa-1)!!\,(C_{\xi,1})^{\kappa}
\,\max_{\ssy D_{\star}}|\lambda|^{2\kappa}\,\max_{[0,\tmax]}\left(2+|f_{_0}|\right)^{\kappa}$.
%
%
Now, combining \eqref{DiscreteM_2}, \eqref{DiscreteM_3} and
\eqref{DiscreteM_4} we obtain
\begin{equation}\label{DiscreteM_5}
\expected\left[\,|\bbg_{n,\ell}|^{2\kappa}\,\right]\leq\,C^{{\ssy
I,D}}_{\kappa}+C^{{\ssy I\!I,D}}_{\kappa}\,\sum_{m=0}^{n-1}\Delta{t_m}
\,\expected\left[\,|\bbg_{m,\ell_m}|^{2\kappa}\,\right]
+C^{{\ssy
I\!I\!I,D}}_{\kappa}\,\left(\,\sum_{m=0}^{n-1} \Delta{t_m}
\,\expected\left[\,|\bbg_{m,\ell_m}|^{2}\,\right]
\,\right)^{\kappa},
\end{equation}
where $C^{{\ssy I,D}}_{\kappa}$, $C^{{\ssy I\!I,D}}_{\kappa}$ and
$C^{{\ssy I\!I\!I,D}}_{\kappa}$ are constants that depend on $J$,
$\kappa$, $\tmax$, $C_{{\ssy D},1,\kappa}$ and $C_{{\ssy D},2,\kappa}$.
\par
First, let us consider the case $\kappa=1$. Then, setting
$\ell=\ell_n$ in \eqref{DiscreteM_5}, we obtain
\begin{equation}\label{DiscreteM_66}
\expected\left[\,|\bbg_{n,\ell_n}|^2\,\right]\leq\,C^{{\ssy
I,D}}_{1}+C^{{\ssy I\!V,D}}_{1}\,\sum_{m=0}^{n-1}\Delta{t_m}
\,\expected\left[\,|\bbg_{m,\ell_m}|^{2}\,\right],\quad
n=1,\dots,N,
\end{equation}
where $C^{{\ssy I\!V,D}}_{1}=C^{{\ssy I\!I,D}}_{1}+C^{{\ssy
I\!I\!I,D}}_{1}$. Setting $\beta_n:=\tfrac{1}{C^{{\ssy
I,D}}_{1}}\,\expected\left[\,|\bbg_{n,\ell_n}|^2\,\right]$ for
$n=0,\dots,N$, \eqref{DiscreteM_66} is written equivalently as follows
\begin{equation}\label{Jeddah2012a}
\beta_n\leq 1+ C_1^{{\ssy I\!V,D}}\,\sum_{m=0}^{n-1}\Delta{t}_m\,\beta_m,
\quad n=1,\dots,N.
\end{equation}
Now, setting $\rho_1:=1$ and $\rho_n:=1+C^{{\ssy
I\!V,D}}_{1}\,\sum_{m=1}^{n-1}\Delta{t_m}\,\rho_m$ for
$n=2,\dots,N$ and observing that $\beta_0=0$, we use
\eqref{Jeddah2012a} and apply a simple induction argument to get
\begin{equation}\label{DiscreteM_6}
\beta_n\leq\rho_n,\quad n=1,\dots,N.
\end{equation}
Since $\rho_n=(1+C^{{\ssy
I\!V,D}}_{1}\,\Delta{t_{n-1}})\,\rho_{n-1}$ for $n=2,\dots,N$, we
use the inequality $e^x\ge1+x$ for $x\ge0$, and a simple induction
argument to conclude that
\begin{equation}\label{DiscreteM_7}
\rho_n\leq\,\exp(C^{{\ssy I\!V,D}}_{1}\,t_{n}),\quad n=1,\dots,N.
\end{equation}
Thus, \eqref{DiscreteM_6} and \eqref{DiscreteM_7} yield
\begin{equation}\label{DiscreteM_8}
\max_{0\leq{m}\leq{\ssy N}}\expected\left[\,
|\bbg_{m,\ell_m}|^2\,\right]\leq\,C^{{\ssy I,D}}_{1}\,\exp(C^{{\ssy
I\!V,D}}_{1}\,\tmax),
\end{equation}
which, along with \eqref{DiscreteM_5}, establishes
\eqref{DiscreteM_0} for $\kappa=1$.
\par
Now, we assume that $\kappa\ge2$. Then, we combine
\eqref{DiscreteM_5} and \eqref{DiscreteM_8} to obtain
\begin{equation}\label{DiscreteM_55}
\expected\left[\,|\bbg_{n,\ell_n}|^{2\kappa}\,\right]\leq\,C^{{\ssy
V,D}}_{\kappa}+C^{{\ssy I\!I,D}}_{\kappa}\,\sum_{m=0}^{n-1}\Delta{t_m}
\,\expected\left[\,|\bbg_{m,\ell_m}|^{2\kappa}\,\right],\quad
n=1,\dots,N,
\end{equation}
where $C^{{\ssy V,D}}_{\kappa}=C^{{\ssy I,D}}_{\kappa}+C^{{\ssy
I\!I\!I,D}}_{\kappa}\,(\tmax)^{\kappa}\,(C^{\ssy
I,D}_1\,\exp(C^{{\ssy I\!V,D}}_1\,\tmax))^{\kappa}$. Then,
proceeding as in obtaining \eqref{DiscreteM_8} from
\eqref{DiscreteM_66}, we arrive at
\begin{equation}\label{DiscreteM_67}
\max_{0\leq{m}\leq{\ssy N}}\expected\left[\,
|\bbg_{m,\ell_m}|^{2\kappa}\,\right]\leq\,C^{{\ssy
V,D}}_{1}\,\exp(C^{{\ssy I\!I,D}}_{1}\,\tmax),
\end{equation}
which, along with \eqref{DiscreteM_5} and \eqref{DiscreteM_8},
yields \eqref{DiscreteM_0} for $\kappa\ge2$.
\end{proof}
\subsection{Estimates for the consistency error}
%
%
In Lemmas~\ref{CONLM_1} and ~\ref{CONLM_2} below, we show that
some Lipschitz-type properties for the solution $g$ to the problem
\eqref{eq:1.5a}--\eqref{eq:1.5b} hold.
%
%
%
%
%
%
\begin{lem}\label{CONLM_1}
Let $\kappa\in{\mathbb N}$ and $g$ be the solution of
\eqref{eq:1.5a}--\eqref{eq:1.5b}. Then, it holds that
\begin{equation}\label{LipTau1}
\expected\left[\,|g(t,\tau_1)-g(t,\tau_2)|^{2\kappa}\,\right]\leq\,C^{\ssy
M}_{\kappa,1}\,
|\tau_1-\tau_2|^{2\kappa},\quad\forall\,\tau_1,\tau_2\in[0,\tau_{\max}],
\quad\forall\,t\in[0,t_{\max}],
\end{equation}
where $C^{\ssy M}_{\kappa,1}$ is the constant in \eqref{KSI_MaximusM_1} for $\ell=1$.
\end{lem}
%
%
%
%
%
%
%
%
%
%
\begin{proof}
Let $t\in[0,\tmax]$ and $\tau_1$, $\tau_2\in[0,\tau_{\max}]$ with
$\tau_2\ge\tau_1$. Then, applying the H{\"o}lder inequality, we
have
\begin{equation*}
\begin{split}
\expected\left[\,|g(t,\tau_1)-g(t,\tau_2)|^{2\kappa}\,\right]=&\,
\expected\left[\,\left|\int_{\tau_1}^{\tau_2}
\partial_{\tau}g(t,\tau)\;d\tau\right|^{2\kappa}\,\right]\\
\leq&\,|\tau_2-\tau_1|^{2\kappa-1}\,\int_{\tau_1}^{\tau_2}
\expected\left[|\partial_{\tau}g(t,\tau)|^{2\kappa}\,\right]\;d\tau\\\
\leq&\,|\tau_2-\tau_1|^{2\kappa}\,\max_{\tau\in[\tau_1,\tau_2]}
\expected\left[|\partial_{\tau}g(t,\tau)|^{2\kappa}\,\right].\\
\end{split}
\end{equation*}
Thus, we obtain \eqref{LipTau1} combining the inequality above and
\eqref{KSI_MaximusM_1} for $\ell=1$.
\end{proof}
\begin{lem}\label{CONLM_2}
Let $\kappa\in{\mathbb N}$ and $g$ be the solution of
\eqref{eq:1.5a}--\eqref{eq:1.5b}. Then, there exists a nonnegative
constant $C_{\ssy {\rm Lip}}$, depending on $\kappa$, $J$,
$\lambda$, ${\widetilde\lambda}$, $f_{_0}$, $C_{\xi,1}$,
$\tau_{\max}$ and $t_{\max}$, such that
\begin{equation}\label{LipT1}
\expected\left[\,|g(t_1,\tau)-g(t_2,\tau)|^{2\kappa}\,\right]\leq\,C_{\ssy
{\rm
Lip}}\,|t_1-t_2|^{\kappa},\quad\forall\,t_1,t_2\in[0,t_{\max}],
\quad\forall\,\tau\in[0,\tau_{\max}].
\end{equation}
\end{lem}
%
%
%
%
%
%
%
%
%
%
\begin{proof}
Let $D_{\star}\equiv[0,t_{\max}]\times[0,\tau_{\max}]$,
$\tau\in[0,\tau_{\max}]$ and $t_1$, $t_2\in[0,\tmax]$ with $t_2\ge
t_1$. Proceeding as in the proof of Lemma~\ref{AUXLM_2} we obtain
\begin{equation}\label{LipT2}
\expected\left[\,|g(\tau,t_1)-g(\tau,t_2)|^{2\kappa}\,\right]
\leq\,(J+1)^{2\kappa-1}\,\left[\,B_{\kappa,\ssy
I}(\tau;t_1,t_2)+B_{\kappa,\ssy I\!I}(\tau;t_1,t_2)\,\right]
\end{equation}
where
\begin{equation*}
\begin{split}
B_{\kappa,\ssy
I}(\tau;t_1,t_2)=&\,\expected\left[\left(\int_{t_1}^{t_2}{\widetilde\lambda}(s,\tau)
\,\xi^2(g(s,s)+f_{_0}(s))\;ds\right)^{2\kappa}\right],\\
B_{\kappa,\ssy
I\!I}(\tau;t_1,t_2)=&\,(2\kappa-1)!!\,\sum_{j=1}^{\ssy
J}\,\left(\,\int_{t_1}^{t_2}(\lambda_j(s,\tau))^2
\,\expected\left[\,\ksi^2(g(s,s)+f_{_0}(s))\,\right]\;ds\right)^{\kappa}.\\
\end{split}
\end{equation*}
Using the H{\"o}lder inequality and \eqref{KSI_1} we obtain
\begin{equation}\label{LipT3}
\begin{split}
B_{\kappa, {\ssy I}}(\tau;t_1,t_2)
\leq&\,(2\,C_{\xi,1})^{2\kappa}\,\max\limits_{\ssy D_{\star}}
|{\widetilde\lambda}|^{2\kappa}\,\expected\left[
\,\left(\int_{t_1}^{t_2}(1+|f_{_0}(s)|)\;ds\right)^{2\kappa}
+\left(\int_{t_1}^{t_2}|g(s,s)|\;ds\right)^{2\kappa}\right]\\
\leq&\,(2\,C_{\xi,1})^{2\kappa}\,\max\limits_{\ssy D_{\star}}
|{\widetilde\lambda}|^2\,|t_1-t_2|^{2\kappa}\,
\,\left(\max_{s\in[t_1,t_2]}(1+|f_{_0}(s)|)^{2\kappa}
+\max_{s\in[t_1,t_2]}\expected\left[\,|g(s,s)|^{2\kappa}\,\right]\right)\\
\end{split}
\end{equation}
and
\begin{equation}\label{LipT4}
\begin{split}
B_{\kappa, {\ssy
I\!I}}(\tau;t_1,t_2)\leq&\,(2\kappa-1)!!\,(C_{\xi,1})^{\kappa}\,\left(\sum_{j=1}^{\ssy
J}\max_{\ssy D_{\star}}|\lambda_j|^{2\kappa}\right)
\,\left(\int_{t_1}^{t_2}\left(2+|f_{_0}(s)|+\expected\left[\,|g(s,s)|^2
\,\right]\,\right)\;ds\right)^{\kappa}\\
&\hskip-2.0truecm\leq\,(2\kappa-1)!!\,(C_{\xi,1})^{\kappa}\,\left(\sum_{j=1}^{\ssy
J}\max_{\ssy
D_{\star}}|\lambda_j|^{2\kappa}\right)\,|t_2-t_1|^{\kappa}\,
\,\max_{s\in[t_1,t_2]}\left(2+|f_{_0}(s)|+\expected\left[\,|g(s,s)|^2
\,\right]\,\right)^{\kappa}.\\
\end{split}
\end{equation}
Thus, \eqref{LipT1} follows easily from \eqref{LipT2},
\eqref{LipT3}, \eqref{LipT4} and \eqref{KSI_MaximusM_1} for
$\ell=0$.
\end{proof}
%
%
%
In Proposition~\ref{CONPRO_1} that follows, we prove a consistency
result for the {\rm (EFD)} and {\rm(EFE)} methods defined in
Section~\ref{sec:2}.
%
%
%
%
\begin{proposition}\label{CONPRO_1}
Let $\kappa\in{\mathbb N}$, $g$ be the solution of
\eqref{eq:1.5a}--\eqref{eq:1.5b},
%
%
$f_{_0}^m:=f_{_0}(t_m)$ for $m=0,\dots,N$, and
${\mathcal K}_{n,\ell}$ be defined by
\begin{equation}\label{DiscreteMC_1}
{\widehat g}_{n+1,\ell}={\widehat
g}_{n,\ell}+\Delta{t_n}\,\nu^{n,\ell}
\,\ksi^2({\widehat g}_{n,\ell_n}+f_{_0}^n)
+\sum_{j=1}^{\ssy J}\mu_j^{n,\ell}\,\ksi({\widehat
g}_{n,\ell_n}+f_{_0}^n)\,\Delta
W_n^j+{\mathcal K}_{n,\ell},
\end{equation}
for $n=0,\dots,N-1$ and $\ell=0,\dots,L-1$, where
$\nu^{n,\ell}={\widetilde\lambda}(t_n,\tau_{\ell})$,
$\mu_j^{n,\ell}=\lambda_j(t_m,\tau_{\ell})$ and ${\widehat
g}_{m,\ell}=g(t_m,\tau_{\ell})$ for the {\rm(EFD)} method, and
$\nu^{m,\ell}=\Pi{\widetilde\lambda}(t_m;\tau_{\ell})$ and
$\mu_j^{m,\ell}=\Pi\lambda_j(t_m;\tau_{\ell})$ and ${\widehat
g}_{m,\ell}=\Pi{g(t_m;\tau_{\ell})}$ for the {\rm(EFE)} method.
Also, we assume that $f_{_0}\in C^1([0,\taumax];\rset)$ and
$\partial_t{\widetilde\lambda}$, $\partial_{t\tau}{\widetilde\lambda}$, 
$(\partial_t\lambda_j)_{j=1}^{\ssy J}$, 
$(\partial_{t\tau}\lambda_j)_{j=1}^{\ssy J}$ 
are well-defined and
continuous on $[0,\tmax]\times[0,\taumax]$. Then, there exists a
nonnegative constant $C_{cn,1}$, independent of the
partitions of the intervals $[0,\tmax]$ and $[0,\taumax]$, such
that
\begin{equation}\label{DiscreteMC_0}
\expected\left[\,\left|\,\sum_{m=0}^n{\mathcal
K}_{m,\ell}\,\right|^{2\kappa}\,\right]\leq\,C_{cn,1}\,\left[\,(\Delta{t})^{\kappa}+(\Delta{\tau})^{2\kappa}\,\right]
\end{equation}
for $n=0,\dots,N-1$ and $\ell=0,\dots,L-1$.
In addition, for the {\rm (EFD)} method there exists a
nonnegative constant $C_{cn,2}$, independent of the
partitions of the intervals $[0,\tmax]$ and $[0,\taumax]$, such
that
\begin{equation}\label{DiscreteMC_000}
\expected\left[\,\left|\,\sum_{m=0}^n{\mathcal
K}_{m,\ell+1}-{\mathcal K}_{m,\ell}\,\right|^{2\kappa}\,\right]\leq\,C_{cn,2}
\,(\Delta\tau_{\ell})^{2\kappa}
\,\left[\,(\Delta{t})^{\kappa}+(\Delta{\tau})^{2\kappa}\,\right]
\end{equation}
for $n=0,\dots,N-1$ and $\ell=0,\dots,L-2$.
\end{proposition}
%
%
%
%
%
%
%
%
%
\begin{proof}
Here, we set $D_{\star}:=[0,\tmax]\times[0,\taumax]$ and use the
symbol $C$ for a generic constant independent of the partitions of
the intervals $[0,\tmax]$ and $[0,\taumax]$.
First, we observe that \eqref{eq:1.5a} yields that
\begin{equation}\label{DiscreteMC_2}
\begin{split}
{\widehat g}_{n+1,\ell}={\widehat
g}_{n,\ell}&\,+\int_{t_n}^{t_{n+1}}\nu^{\ell}(s)
\,\ksi^2(g(s,s)+f_{_0}(s))\;ds\\
&\,+\sum_{j=1}^{\ssy
J}\int_{t_n}^{t_{n+1}}\mu_j^{\ell}(s)\,\ksi(g(s,s)+f_{_0}(s))\;dW^j(s),\\
\end{split}
\end{equation}
for $n=0,\dots,N-1$ and $\ell=0,\dots,L-1$, where
$\nu^{\ell}(s)={\widetilde\lambda}(s,\tau_{\ell})$  and
$\mu_j^{\ell}(s)=\lambda_j(s,\tau_{\ell})$ for the {\rm (EFD)}
method and $\nu^{\ell}(s)=\Pi{\widetilde\lambda}(s;\tau_{\ell})$
and $\mu_j^{\ell}(s)=\Pi\lambda_j(s;\tau_{\ell})$ for the {\rm
(EFE)} method.
Then, subtracting \eqref{DiscreteMC_2} from \eqref{DiscreteMC_1}
we obtain
%
$\sum_{m=0}^{n}{\mathcal K}_{m,\ell}=\sum_{i=1}^4 E_{i,{\ssy
C}}^{n,\ell}$
%
for $\ell=0,\dots,L-1$ and $n=0,\dots,N-1$, where
\begin{equation*}
\begin{split}
E_{1,{\ssy C}}^{n,\ell}\equiv&\,\sum_{m=0}^{n}\int_{t_m}^{t_{m+1}}
(\nu^{\ell}(s)
-\nu^{m,\ell})\,\,\ksi^2(g(s,s)+f_{_0}(s))\;ds,\\
E_{2,{\ssy C}}^{n,\ell}\equiv&\,\sum_{m=0}^{n}\int_{t_m}^{t_{m+1}}
\nu^{m,\ell}\,\left[\,
\ksi^2(g(s,s)+f_{_0}(s))-\ksi^2({\widehat g}_{m,\ell_m}+f_{_0}^m)\,\right]\;ds,\\
E_{3,{\ssy C}}^{n,\ell}\equiv&\,\sum_{j=1}^{\ssy
J}\sum_{m=0}^{n}\int_{t_m}^{t_{m+1}}(\mu_j^{\ell}(s)-\mu_j^{n,\ell})
\,\,\ksi(g(s,s)+f_{_0}(s))\;dW^j(s),\\
E_{4,{\ssy C}}^{n,\ell}\equiv&\,\sum_{j=1}^{\ssy
J}\sum_{m=0}^{n}\int_{t_m}^{t_{m+1}}\mu_j^{n,\ell}\,\left[\,
\ksi(g(s,s)+f_{_0}(s))-\ksi({\widehat g}_{m,\ell_m}+f_{_0}^m)
\,\right]\;dW^j(s).\\
\end{split}
\end{equation*}
Next, using \eqref{KSI_1}, the H{\"o}lder inequality and
\eqref{KSI_MaximusM_1}, we obtain
\begin{equation}\label{DiscreteMC_4}
\begin{split}
\expected\left[(E_{1,{\ssy C}}^{n,\ell})^{2\kappa}\right]
\leq&\,(C_{\xi,1})^{2\kappa}\,(\Delta{t})^{2\kappa}\,\max_{\ssy
D_{\star}}|\partial_t{\widetilde\lambda}|^{2\kappa}\,\expected\left[\,\left(
\int_{0}^{\tmax}(1+|f_{_0}(s)|+|g(s,s)|)\;ds\right)^{2\kappa}\,\right]\\
\leq&\,C\,(\Delta{t})^{2\kappa}\,\expected\left[\,
\left(\int_0^{\tmax}(1+|f_{_0}(s)|\;ds\right)^{2\kappa}
+\left(\int_0^{\tmax}|g(s,s)|\;ds\right)^{2\kappa}\right]\\
\leq&\,C\,(\Delta{t})^{2\kappa}\,\left[\,
1+(\tmax)^{2\kappa}\,\max_{s\in[0,\tmax]}
\expected\left[\,(g(s,s))^{2\kappa}\,\right]\right]\\
\leq&\,C\,(\Delta{t})^{2\kappa}\,\left[\,
1+(\tmax)^{2\kappa}\,C^{\ssy M}_{\kappa,0}\,\right]\\
\end{split}
\end{equation}
and
\begin{equation}\label{DiscreteMC_6}
\begin{split}
\expected\left[\,(E^{n,\ell}_{3,{\ssy
C}})^{2\kappa}\,\right]\leq&\,C\,\sum_{j=1}^{\ssy
J}\,\expected\left[\,\sum_{m=0}^{n}\int_{t_m}^{t_{m+1}}
(\mu_j^{\ell}(s) -\mu_j^{m,\ell})
\,\ksi(g(s,s)+f_{_0}(s))\;dW^j(s)\,\right]^{2\kappa}\\
\leq&\,C\,\sum_{j=1}^{\ssy
J}\,\left[\,\sum_{m=0}^{n}\int_{t_m}^{t_{m+1}} (\mu_j^{\ell}(s)
-\mu_j^{m,\ell})^2
\,\expected\left[\,\ksi^2(g(s,s)+f_{_0}(s))\,\right]\;ds\,\right]^{\kappa}\\
\leq&\,C\,(\Delta{t})^{2\kappa}\,\left(\,\int_0^{\tmax}\left(\,1+|f_{_0}(s)|
+\expected\left[\,|g(s,s)|\,\right]\,\right)\;ds\,\right)^{\kappa}\\
\leq&\,C\,(\Delta
t)^{2\kappa}\,\left(\,\int_0^{\tmax}\left(\,2+|f_{_0}(s)|
+\expected\left[\,(g(s,s))^{2}\,\right]
\,\right)\;ds\,\right)^{\kappa}\\
\leq&\,C\,(\Delta
t)^{2\kappa}\,\left(\,\int_0^{\tmax}\left(\,2+|f_{_0}(s)| +C^{\ssy
M}_{1,0}
\,\right)\;ds\,\right)^{\kappa}\\
\end{split}
\end{equation}
Now, we apply \eqref{KSI_2} and the
H{\"o}lder inequality, to get
\begin{equation}\label{DiscreteMC_5}
\begin{split}
\expected\left[(E_{2,{\ssy
C}}^{n,\ell})^{2\kappa}\right]\leq&\,C\,\expected\left[\,
\left(\,\sum_{m=0}^{n}\int_{t_m}^{t_{m+1}}\left(|g(s,s)-{\widehat
g}_{m,\ell_m}|
+|f_{_0}(s)-f_{_0}^m|\right)\;ds\right)^{2\kappa}\,\right]\\
\leq&\,C\,\expected\left[\,
\left(\,\sum_{m=0}^{n}\int_{t_m}^{t_{m+1}}|f_{_0}(s)-f_{_0}^m|\;ds\right)^{2\kappa}
+\left(\,\sum_{m=0}^{n}\int_{t_m}^{t_{m+1}}|g(s,s)-{\widehat
g}_{m,\ell_m}|
\;ds\right)^{2\kappa}\right]\\
\leq&\,C\,\left[\,\sum_{m=0}^{n}\int_{t_m}^{t_{m+1}}
|f_{_0}(s)-f_{_0}^m|^{2\kappa}\;ds
+\sum_{m=0}^{n}\int_{t_m}^{t_{m+1}}\expected\left[
|g(s,s)-{\widehat g}_{m,\ell_m}|^{2\kappa}\,\right]\;ds\right]\\
\leq&\,C\,\left[\,(\Delta{t})^{2\kappa}
\,\max_{[0,\taumax]}|f_{_0}'|^{2\kappa}
+\,\sum_{m=0}^{n}\int_{t_m}^{t_{m+1}}\expected\left[
|g(s,s)-{\widehat g}_{m,\ell_m}|^{2\kappa}\,\right]\;ds\right]\\
\end{split}
\end{equation}
and
\begin{equation}\label{DiscreteMC_7}
\begin{split}
\expected\left[\,(E_{4,{\ssy C}}^{n,\ell})^{2\kappa}\,\right]
\leq&\,C\,\sum_{j=1}^{\ssy
J}\,\expected\left[\,\sum_{m=0}^{n}\int_{t_m}^{t_{m+1}}
\mu_j^{n,\ell}\,\left[\,\ksi(g(s,s)+f_{_0}(s))
-\ksi({\widehat g}_{m,\ell_m}+f_{_0}^m)\right]\;dW^j(s)\,\right]^{2\kappa}\\
\leq&\,C\,\left[\,\sum_{m=0}^{n}\int_{t_m}^{t_{m+1}}(\mu_j^{n,\ell})^2\,
\expected\left[\,\left(\ksi(g(s,s)+f_{_0}(s)) -\ksi({\widehat
g}_{m,\ell_m}+f_{_0}^m)\right)^2\,\right]
\;ds\,\right]^{\kappa}\\
\leq&\,C\,\left[\sum_{m=0}^{n}\int_{t_m}^{t_{m+1}}\,\left(\,
\expected\left[|g(s,s)-{\widehat g}_{m,\ell_m}|^2\,\right]
+|f_{_0}(s)-f_{_0}^m|^2\,\right)\;ds\right]^{\kappa}\\
\leq&\,C\,\left[(\Delta{t})^2\,\max_{[0,\taumax]}|f_{_0}'|^2
+\sum_{m=0}^{n}\int_{t_m}^{t_{m+1}}\,
\expected\left[\,|g(s,s)-{\widehat g}_{m,\ell_m}|^2\,\right]\right]^{\kappa}\\
\end{split}
\end{equation}
 Using \eqref{LipTau1}, \eqref{LipT1} and
\eqref{eq:A.4}, we have
\begin{equation}\label{DiscreteMC_77}
\begin{split}
\expected\left[\,|g(s,s)-{\widehat
g}_{m,\ell_m}|^{2\kappa}\,\right]\leq&\,C\,\Big(
\expected\left[\,|g(s,s)-g(s,\tau_{\ell_m})|^{2\kappa}\,\right]
+\expected\left[\,|g(s,\tau_{\ell_m})
-g(t_m,\tau_{\ell_m})|^{2\kappa}\,\right]\\
&\quad\quad+\expected\left[\,|g(t_m,\tau_{\ell_m})-{\widehat
g}_{m,\ell_m}|^{2\kappa}\,\right]\,\Big)\\
\leq&\,C\,\left(\,|s-\tau_{\ell_m}|^{2\kappa}+|s-t_m|^{\kappa}
+\expected\left[\,|g(t_m,\tau_{\ell_m})-{\widehat
g}_{m,\ell_m}|^{2\kappa}\,\right]\,\right)\\
\leq&\,C\,\left(|s-t_m|^{2\kappa}+|t_m-\tau_{\ell_m}|^{2\kappa}+(\Delta{t})^{\kappa}
+\expected\left[\,|g(t_m,\tau_{\ell_m})-{\widehat
g}_{m,\ell_m}|^{2\kappa}\,\right]\right)\\
\leq&\,C\,\left(|\tau_{\ell_m+1}-\tau_{\ell_m}|^{2\kappa}+(\Delta{t})^{\kappa}
+\expected\left[\,|g(t_m,\tau_{\ell_m})-{\widehat
g}_{m,\ell_m}|^{2\kappa}\,\right]\right)\\
\leq&\,C\,\left((\Delta\tau)^{2\kappa}+(\Delta{t})^{\kappa}
+\expected\left[\,|g(t_m,\tau_{\ell_m})-{\widehat
g}_{m,\ell_m}|^{2\kappa}\,\right]\right)\\
\end{split}
\end{equation}
for $s\in[t_m,t_{m+1}]$ and $m=0,\dots,N-1$. For the (EFE) method,
after using \eqref{LipTau1}, we have
\begin{equation}\label{DiscreteMC_78}
\begin{split}
\expected\left[\,|g(t_m,\tau_{\ell_m})-{\widehat
g}_{m,\ell_m}|^{2\kappa}\,\right]\leq&\,(\Delta\tau_{\ell_m})^{-1}\,
\int_{\tau_{\ell_m}}^{\tau_{\ell_m+1}}\expected\left[
\,(g(t_m,\tau_{\ell_m})-g(t_m,\tau))^{2\kappa}\,\right]\;d\tau\\
\leq&\,C\,(\Delta\tau_{\ell_m})^{-1}\,
\int_{\tau_{\ell_m}}^{\tau_{\ell_m+1}}|\tau_{\ell_m}-\tau|^{2\kappa}\;d\tau\\
\leq&\,C\,(\Delta\tau)^{2\kappa},\quad m=0,\dots,N-1,\\
\end{split}
\end{equation}
while for the (EFD) method the term we estimate above vanishes.
Finally, \eqref{DiscreteMC_77} and \eqref{DiscreteMC_78} yield
\begin{equation}\label{DiscreteMC_79}
\expected\left[\,|g(s,s)-{\widehat
g}_{m,\ell_m}|^{2\kappa}\,\right]\leq\,C\,\left[(\Delta
t)^{\kappa}+(\Delta\tau)^{2\kappa}\right]
\end{equation}
for $s\in[t_m,t_{m+1}]$ and $m=0,\dots,N-1$.
Observing that
$\expected\left[\left(\sum_{\ell=0}^{n}{\mathcal
K}_{m,\ell}\right)^{2\kappa}\right]\leq\,4^{2\kappa-1}
\sum_{i=1}^4\expected\left[(E^{n,\ell}_{i,{\ssy
C}})^{2\kappa}\right]$ for $\ell=0,\dots,L-1$ and $n=0,\dots,N-1$,
and that estimate \eqref{DiscreteMC_79} holds for $\kappa=1$, the
estimate \eqref{DiscreteMC_0} for the consistency error follows
easily in view of \eqref{DiscreteMC_4}, \eqref{DiscreteMC_6},
\eqref{DiscreteMC_5}, \eqref{DiscreteMC_7} and
\eqref{DiscreteMC_79}.
\par
Since, $\expected\left[\left(\sum_{\ell=0}^{n}{\mathcal
K}_{m,\ell+1}-{\mathcal
K}_{m,\ell+1}\right)^{2\kappa}\right]\leq\,4^{2\kappa-1}
\sum_{i=1}^4\expected\left[(E^{n,\ell+1}_{i,{\ssy
C}}-(E^{n,\ell}_{i,{\ssy
C}})^{2\kappa}\right]$
for $\ell=0,\dots,L-2$ and $n=0,\dots,N-1$, we 
obtain  \eqref{DiscreteMC_000} for the (EFD) nethod, observing that
\begin{equation}\label{aytes_oi_anisothtes}
\begin{split}
|y^{m,\ell+1}-y^{m,\ell}|\leq&\,C\,\Delta{\tau}_{\ell}\\
\left|(y^{\ell+1}(s)-y^{m,\ell+1})-(y^{\ell}(s)-y^{m,\ell})\right|
=&\,\left|\int_{t_m}^s\int_{\tau_{\ell}}^{\tau_{\ell+1}}
\partial_{t\tau}y(t',\tau')dt'd\tau'\right|\\
\leq&\,C\,\Delta{t}\,\Delta{\tau}_{\ell},\quad\forall\,s\in[t_m,t_{m+1}],
\end{split}
\end{equation}
where $y^{m,\ell}=\nu^{m,\ell}$ or $\mu_j^{n,\ell}$ and $y={\widetilde\lambda}$
or $\lambda_j$, respectively, and proceeding as above.
\end{proof}
%
%
%
%
%
%
%
%
%
\subsection{Error estimation}
%
%
%
%
In this section we derive an error estimate for the strong approximation error
${\mathcal G}(g)-{\overline{\mathcal G}}(\bbg)$ by splitting it as sum of
the strong  discretization error ${\mathcal G}(g)-{\mathcal G}(\bbg)$
which we estimate in Theorem~\ref{A_basic_theorem} and of  the strong
mumerical quadrature error ${\mathcal G}(\bbg)-{\overline{\mathcal G}}(\bbg)$
which we estimate in Theorem~\ref{Theorem_Q}.
%
%
%
\begin{thm}\label{A_basic_theorem}
Let $g$, $Y$ and $Z$ be the solution of
\eqref{eq:1.5a}--\eqref{eq:1.5b}, 
${\mathcal M}:=\{0,\dots,N\}\times\{0,\dots,L-2\}$
${\mathcal
J}:=\{0,\dots,N\}\times\{0,\dots,L-1\}$, ${\mathcal
I}:=\{0,\dots,N\}\times\{0,\dots,L+1\}$ and
$(\bbg_{n,\ell})_{(n,\ell)\in{\mathcal I}}$ be the
numerical approximations produced by the \text{\rm(EFD)} or the
\text{\rm (EFM)} method.
Also, we assume that the functions $\Psi'$, $F$, $F'$, $G$, $G'$, $U$, $U':{\mathbb
R}\rightarrow{\mathbb R}$ have polynomial growth, and we define
$\Lambda_{\ssy\Psi}(w):=\Lambda(\Psi(w+f_{_0}))$
for $w\in S_{\ssy\Delta\tau}$ or $w\in C([0,\taumax];{\mathbb R})$.
Then, there exist nonnegative constants $(C_{i}^{\ssy C\!V})_{i=1}^6$,
independent of the partitions of the
intervals $[0,\tmax]$ and $[0,\taumax]$, such that
\begin{equation}\label{DiscreteMEE_0}
\max_{(n,\ell)\in{\mathcal J}}\left(\expected\left[\,
\left|g(t_n,\tau_{\ell})-\bbg_{n,\ell}\right|^{2\kappa}
\,\right]\right)^{\frac{1}{2\kappa}}\leq\,C_1^{\ssy
C\!V}\,\big[\,(\Delta{t})^{\frac{1}{2}}+\Delta\tau\,\big],
\end{equation}
\begin{equation}\label{DiscreteMEE_00}
\max_{0\leq{n}\leq{\ssy N}}\left(\expected\left[\,|Y(t_n)-\bbg_{n,
\ssy L}|^{2\kappa}\,\right]\right)^{\frac{1}{2\kappa}}\leq\,C_2^{\ssy
C\!V}\,\big[\,(\Delta{t})^{\frac{1}{2}}
+\Delta\tau\,\big],
\end{equation}
\begin{equation}\label{DiscreteMEE_000}
\left(\expected\left[\,|Z(\tmax)-\bbg_{\ssy N,
L+1}|^{2\kappa}\,\right]\right)^{\frac{1}{2\kappa}}\leq\,C_3^{\ssy
C\!V}\,\big[\,(\Delta{t})^{\frac{1}{2}}
+\Delta\tau\,\big],
\end{equation}
\begin{equation}\label{DiscreteMEE_0000}
\left(\expected\left[\,\left|\,\Lambda_{\ssy\Psi}(g)-\Lambda_{\ssy\Psi}(\bbg)\,\right|^{2\kappa}
\,\right]\right)^{\frac{1}{2\kappa}}\leq\,C_4^{\ssy C\!V}\,\big[\,(\Delta{t})^{\frac{1}{2}}
+\Delta\tau\,\big],
\end{equation}
\begin{equation}\label{DiscreteMEE_Functional}
\left(\expected\left[\,|{\mathcal G}(g)-{\mathcal G}(\bbg)|^{2\kappa}
\,\right]\right)^{\frac{1}{2\kappa}}\leq\,C_5^{\ssy C\!V}\,\big[\,(\Delta{t})^{\frac{1}{2}}
+\Delta\tau\,\big]
\end{equation}
and, for the {\rm (EFD)} method, 
\begin{equation}\label{traviata_0}
\max_{(n,\ell)\in{\mathcal M}}\left(\expected\left[\,
\left|\tfrac{g(t_n,\tau_{\ell+1})-g(t_n,\tau_{\ell})}{\Delta\tau_{\ell}}
-\tfrac{\bbg_{n,\ell+1}-\bbg_{n,\ell}}{\Delta\tau_{\ell}}\right|^{2\kappa}
\,\right]\right)^{\frac{1}{2\kappa}}\leq\,C_6^{\ssy
C\!V}\,\big[\,(\Delta{t})^{\frac{1}{2}}+\Delta\tau\,\big].
\end{equation}
\end{thm}
%
%
%
%
\begin{proof}
Here, we set $D_{\star}:=[0,\tmax]\times[0,\taumax]$ and will use
the symbol $C$ for a generic constant independent of the
partitions of the intervals $[0,\tmax]$ and $[0,\taumax]$.
Let $E_{m,\ell}={\widehat g}_{m,\ell}-\bbg_{m,\ell}$ for
$m=0,\dots,N$ and $\ell=0,\dots,L-1$, where ${\widehat
g}_{m,\ell}=g(t_m,\tau_{\ell})$ for the {\rm(EFD)} method and
${\widehat g}_{m,\ell}=\Pi g(t_m;\tau_{\ell})$ for the {\rm(EFE)}
method.
First, subtract \eqref{eq:2.1b} or \eqref{eq:2.2b} from
\eqref{DiscreteMC_1}, and then sum with respect to $n$, to obtain
\begin{equation*}\label{DiscreteMEE_1}
E_{n,\ell}=A_{n,\ell}+B_{n,\ell}+\sum_{m=0}^{n-1}{\mathcal
K}_{m,\ell},\quad \ell=0,\dots,L-1,\quad n=1,\dots,N,
\end{equation*}
where
\begin{equation*}
\begin{split}
A_{n,\ell}:=&\,\sum_{m=0}^{n-1}\Delta
t_m\,\nu^{m,\ell}\,\left[\,\ksi^2({\widehat
g}_{m,\ell_m}+f_{_0}^m)
-\ksi^2(\bbg_{m,\ell_m}+f_{_0}^m)\,\right],\\
B_{n,\ell}:=&\,\sum_{m=0}^{n-1}\sum_{j=1}^{\ssy
J}\,\mu^{m,\ell}_j\,\left[\, \ksi\big({\widehat
g}_{m,\ell_m}+f_{_0}^m\big)
-\ksi\big(\bbg_{m,\ell_m}+f_{_0}^m\big)\,\right]
\,\Delta{W_m^j},\\
\end{split}
\end{equation*}
$f_{_0}^m:=f_{_0}(t_m)$,
$\nu^{m,\ell}={\widetilde\lambda}(t_m,\tau_{\ell})$  and
$\mu_j^{m,\ell}=\lambda_j(t_m,\tau_{\ell})$ for the {\rm (EFD)}
method and $\nu^{m,\ell}=\Pi{\widetilde\lambda}(t_m;\tau_{\ell})$
and $\mu_j^{m,\ell}=\Pi\lambda_j(t_m;\tau_{\ell})$ for the {\rm
(EFE)} method. Thus, we have
\begin{equation}\label{DiscreteMEE_2}
\expected\left[\,|E_{n,\ell}|^{2\kappa}\,\right]=3^{2\kappa-1}
\,\left(\,\expected\left[\,|A_{n,\ell}|^{2\kappa}\,\right]
+\expected\left[\,|B_{n,\ell}|^{2\kappa}\,\right]
+\expected\left[\,\left|\,\sum_{m=1}^{n-1}{\mathcal
K}_{m,\ell}\,\right|^{2\kappa}\,\right]\,\right)
\end{equation}
for $\ell=0,\dots,L-1$ and $n=1,\dots,N$.
First, using \eqref{KSI_2} and the H{\"o}lder inequality, we
obtain
\begin{equation}\label{DiscreteMEE_3}
\begin{split}
\expected\left[\,|A_{n,\ell}|^{2\kappa}\,\right]\leq&\,(C_{\xi,2})^2\,\max_{\ssy
D_{\star}}|{\widetilde\lambda}|^{2\kappa}\,\expected\left[\,
\left(\sum_{m=0}^{n-1}\Delta{t_m}\,
|E_{m,\ell_m}|\right)^{2\kappa}\,\right]\\
\leq&\,C\,\left(\,\sum_{m=0}^{n-1}\Delta{t_m}\,\expected
\left[\,(E_{m,\ell_m})^{2\kappa}\,\right]\,\right)\\
\end{split}
\end{equation}
and
\begin{equation}\label{DiscreteMEE_4}
\begin{split}
\expected\left[\,|B_{n,\ell}|^{2\kappa}\,\right]\leq&\,C\,\sum_{j=1}^{\ssy
J}\left(\,\sum_{m=0}^{n-1}\Delta{t_m}\,(\mu_j^{m,\ell})^2
\,\expected\left[\,\left|\ksi({\widehat g}_{m,\ell_m}+f_{_0}^m)-
\ksi(\bbg_{m,\ell_m}+f_{_0}^m)\right|^2\,\right]\right)^{\kappa}\\
\leq&\,C\,\left(\,\sum_{m=0}^{n-1}\Delta{t_m}
\,\expected\left[\,\left|\ksi({\widehat g}_{m,\ell_m}+f_{_0}^m)-
\ksi(\bbg_{m,\ell_m}+f_{_0}^m)\right|^2\,\right]\right)^{\kappa}\\
\leq&\,C\,\left(\, \sum_{m=0}^{n-1}\,\Delta{t_m}\,\expected\left[
\,(E_{m,\ell_m})^2\,\right]\,\right)^{\kappa}\\
\end{split}
\end{equation}
for $\ell=0,\dots,L-1$ and $n=1,\dots,N$.
Combining, \eqref{DiscreteMEE_2}, \eqref{DiscreteMEE_3} and
\eqref{DiscreteMEE_4} and \eqref{DiscreteMC_0}, we have
\begin{equation}\label{DiscreteMEE_5}
\expected\left[\,|E_{n,\ell}|^{2\kappa}\right]\leq C\,\left[\,\left(
(\Delta{t})^{\kappa}+(\Delta\tau)^{2\kappa}\right)+
\sum_{m=0}^{n-1}\Delta{t_m}\,\expected\left[\,|E_{m,\ell_m}|^{2\kappa}\,\right]
+\left(\sum_{m=0}^{n-1}\Delta{t_m}
\,\expected\left[\,|E_{m,\ell_m}|^{2}
\,\right]\right)^{\kappa}\,\right]
\end{equation}
for $\ell=0,\dots,L-1$ and $n=1,\dots,N$.
Considering the case $\kappa=1$ and proceeding as in the proof of
Lemma~\ref{AUXLM_3}, from \eqref{DiscreteMEE_5} we arrive at the
estimate
\begin{equation}\label{DiscreteMEE_6}
\max_{0\leq{n}\leq{\ssy N}}\expected\left[\,
|E_{n,\ell_n}|^2\,\right]\leq\,C\,(\Delta{t}+(\Delta\tau)^2).
\end{equation}
Letting $\kappa\ge2$, under the view of \eqref{DiscreteMEE_6}, the
inequality \eqref{DiscreteMEE_5} yields
\begin{equation}\label{DiscreteMEE_7}
\expected\left[\,|E_{n,\ell}|^{2\kappa}\,\right]\leq C\,\left[\,\left(
(\Delta{t})^{\kappa}+(\Delta\tau)^{2\kappa}\right)
+\sum_{m=0}^{n-1}\Delta{t_m}\,
\expected\left[\,|E_{m,\ell_m}|^{2\kappa}\,\right]\right]
\end{equation}
for $\ell=0,\dots,L-1$ and $n=1,\dots,N$. Now, proceeding again as
in the proof of Lemma~\ref{AUXLM_3}, from \eqref{DiscreteMEE_7} we
conclude that
\begin{equation}\label{DiscreteMEE_8}
\max_{0\leq{n}\leq{\ssy N}}\expected\left[\,
|E_{n,\ell_n}|^{2\kappa}\,\right]
\leq\,C\,\left((\Delta{t})^{\kappa}+(\Delta\tau)^{2\kappa}\right).
\end{equation}
Thus, combining \eqref{DiscreteMEE_7} and \eqref{DiscreteMEE_8} we
arrive at
\begin{equation}\label{DiscreteMEE_9}
\max_{(n,\ell)\in{\mathcal
I}_{\ssy N,L}}\expected\left[\,|E_{n,\ell}|^{2\kappa}\right]\leq\,C\left(
(\Delta{t})^{\kappa}+(\Delta\tau)^{2\kappa}\right).
\end{equation}
The estimate \eqref{DiscreteMEE_0} for the (EFD) method follows
directly from \eqref{DiscreteMEE_9}. For the (EFE) method,
\eqref{DiscreteMEE_0} follows combining \eqref{DiscreteMEE_9} with
the following estimate (cf. \eqref{DiscreteMC_78})
\begin{equation*}
\begin{split}
\max_{(n,\ell)\in{\mathcal
I}_{\ssy N,L}}\expected\left[\,|g(t_n,\tau_{\ell})-\Pi
g(t_n;\tau_{\ell})|^{2\kappa}\,\right]\leq&\,\max_{(n,\ell)\in{\mathcal
I}}\hskip0.2truecm\max_{\tau\in[\tau_{\ell},\tau_{\ell+1}]}
\expected\left[\,|g(t_m,\tau)-g(t_m,\tau_{\ell})|^{2\kappa}\,\right]\\
\leq&\,C\,(\Delta\tau)^{2\kappa}.\\
\end{split}
\end{equation*}
Since 
$\tfrac{E_{n,\ell+1}-E_{n,\ell}}{\Delta\tau_{\ell}}=\tfrac{A_{n,\ell+1}-A_{n,\ell}}{\Delta\tau_{\ell}}
+\tfrac{B_{n,\ell+1}-B_{n,\ell}}{\Delta\tau_{\ell}}+\sum_{m=0}^{n-1}\tfrac{
{\mathcal K}_{m,\ell+1}-{\mathcal
K}_{m,\ell}}{\Delta\tau_{\ell}}$
for $\ell=0,\dots,L-2$ and $n=1,\dots,N$, to obtain the estimate
\eqref{traviata_0} for the (EFD) method we proceed as above using
\eqref{DiscreteMC_000} and \eqref{aytes_oi_anisothtes}.
\par
In order to get the second error estimate, we use \eqref{eq:1.3b} and \eqref{eq:2.1b} or
\eqref{eq:2.2b}, to conclude that
\begin{equation}\label{DiscreteMEE_55}
\expected\left[\,\left|Y(t_n)-\bbg_{n,L}\right|^{2\kappa}
\,\right]\leq\,C\,\sum_{i=1}^3\expected\left[
\,|\zeta_i^n|^{2\kappa}\,\right],\quad n=1,\dots,N,
\end{equation}
where
\begin{equation*}
\begin{gathered}
\zeta^n_1:=\sum_{m=0}^{\ssy
n-1}\int_{t_m}^{t_{m+1}}(f_{_0}(s)-f_{_0}(t_m))\;ds,\quad
\zeta^n_2:=\sum_{m=0}^{\ssy
n-1}\int_{t_m}^{t_{m+1}}(g(s,s)-g(t_m,\tau_{\ell_m}))\;ds,\\
\zeta^n_3:=\sum_{m=0}^{\ssy
n-1}\Delta{t_m}\,(g(t_m,\tau_{\ell_m})-\bbg_{m,\ell_m}).
\end{gathered}
\end{equation*}
First, we observe that
\begin{equation}\label{DiscreteMEE_56}
|\zeta_1^n|^{2\kappa}\leq\,C\,(\Delta{t})^{2\kappa}
\,\max_{[0,\tmax]}|f_{_0}'|^{2\kappa},
\quad n=1,\dots,N.
\end{equation}
Next, we use the H{\"o}lder inequality, \eqref{LipT1}, \eqref{LipTau1} and
\eqref{DiscreteMEE_0} to obtain
\begin{equation}\label{DiscreteMEE_57}
\begin{split}
\expected\left[\,|\zeta^n_2|^{2\kappa}\,\right]\leq&\,C\,\sum_{m=0}^{\ssy
n-1}\int_{t_m}^{t_{m+1}}\expected\left[
|g(s,s)-g(t_m,\tau_{\ell_m})|^{2\kappa}\right]\;ds\\
\leq&\,C\,\sum_{m=0}^{\ssy
n-1}\int_{t_m}^{t_{m+1}}\expected\left[|g(s,s)-g(t_m,s)|^{2\kappa}
+|g(t_m,s)-g(t_m,\tau_{\ell_m})|^{2\kappa}\right]\;ds\\
\leq&\,C\,\left[\,(\Delta{t})^{\kappa}+\sum_{m=0}^{\ssy
n-1}\int_{t_m}^{t_{m+1}}\left(|s-t_m|^{2\kappa}
+|t_m-\tau_{\ell_m}|^{2\kappa}\right)\;ds\right]\\
\leq&\,C\,\left[\,(\Delta{t})^{\kappa}+\sum_{m=0}^{\ssy
n-1}{\Delta{t_m}}\,|\tau_{\ell_m+1}
-\tau_{\ell_m}|^{2\kappa}\right]\\
\leq&\,C\,\left[\,(\Delta{t})^{\kappa}
+(\Delta\tau)^{2\kappa}\,\right],\quad n=1,\dots,N,
\end{split}
\end{equation}
and
\begin{equation}\label{DiscreteMEE_58}
\begin{split}
\expected\left[\,|\zeta^n_3|^{2\kappa}\right]\leq&\,C\,\sum_{m=0}^{\ssy
n-1}\Delta{t_m}\,\expected\left[|g(t_m,\tau_{\ell_m})
-\bbg_{m,\ell_m}|^{2\kappa}\right]\\
\leq&\,C\,\left[\,(\Delta{t})^{\kappa}+(\Delta\tau)^{2\kappa}\,\right],
\quad n=1,\dots,N.\\
\end{split}
\end{equation}
Thus, \eqref{DiscreteMEE_00} follows easily from
\eqref{DiscreteMEE_55}, \eqref{DiscreteMEE_56},
\eqref{DiscreteMEE_57} and \eqref{DiscreteMEE_58}.
\par
In order to prove our third error estimate, we use \eqref{eq:1.3b},
\eqref{eq:2.1b} or \eqref{eq:2.2b}, and the mean value theorem for
scalar fields, to conclude that
\begin{equation}\label{DiscreteMEE_55}
\expected\left[\,\left|Z(\tmax)-\bbg_{\ssy
N,L+1}\right|^{2\kappa}\,\right]\leq\,C\,\sum_{i=1}^3\expected\left[\,
|\Gamma_{i}|^{2\kappa}\,\right],
\end{equation}
where
\begin{equation*}
\begin{split}
\Gamma_1:=&\,\sum_{m=0}^{\ssy N-1}\int_{t_m}^{t_{m+1}}
F'(A_m(s))\,U(B_m(s))\,(Y(s)-\bbg_{m,{\ssy L}})\;ds,\\
\Gamma_2:=&\,\sum_{m=0}^{\ssy N-1}\int_{t_m}^{t_{m+1}}
F(A_m(s))\,U'(B_m(s))\,(g(s,s)-\bbg_{m,\ell_m})\;ds,\\
\Gamma_3:=&\,\sum_{m=0}^{\ssy N-1}\int_{t_m}^{t_{m+1}}
F(A_m(s))\,U'(B_m(s))\,(f_{_0}(s)-f_{_0}^m)\;ds,\\
\end{split}
\end{equation*}
and
\begin{equation*}
\begin{split}
A_m(s):=&\,\delta_{m}(s)\,\left(Y(s)-\bbg_{m,{\ssy L}}\right)+\,\bbg_{m,{\ssy L}},\\
B_m(s):=&\,{\tilde\delta}_{m}(s)\,\left(\,g(s,s)+f_{_0}(s)\right)
+(1-{\tilde\delta}_m(s))\,(\barrg_{m,\ell_m}+f_{_0}^m),\\
\end{split}
\end{equation*}
with  $\delta_{m}(s)$, ${\tilde\delta}_{m}(s)\in[0,1]$.
Let ${\tilde{m}}\in{\mathbb N}$. Since $F$, $F'$, $U$ and $U'$ have
polynomial growth, we use \eqref{DiscreteMEE_00}, \eqref{Moments_of_Functional},
\eqref{DiscreteM_0}
and \eqref{KSI_MaximusM_1}  to conclude that there exist
a nonnegative constant $C_{\star}^m$ such that
\begin{equation}\label{GlobalB_1}
\max_{0\leq{m}\leq{\ssy N-1}}\sup_{s\in(t_m,t_{m+1})}\left[\,\expected\left[
\,|F'(A_m(s))\,U(B_m(s))|^{2{\tilde{m}}}\,\right]
+\expected\left[\,|F(A_m(s))\,U'(B_m(s))|^{2{\tilde{m}}}\,\right]\,\right]
\leq\,C_{\star}^{{\tilde{m}}}.
\end{equation}
Also, we use the H{\"o}lder inequality and \eqref{KSI_MaximusM_1} to arrive at
\begin{equation}\label{Ymoments}
\begin{split}
\expected\left[\,|Y(t_b)-Y(t_a)|^{2{\tilde{m}}}\,\right]
\leq&\,(t_b-t_a)^{2{\tilde m}-1}\,\int_{t_a}^{t_b}
\expected\left[|g(s,s)+f_{_0}(s)|^{2{\tilde{m}}}
\,\right]\;ds\\
\leq&\,C\,(t_b-t_a)^{2{\tilde m}}\\
\end{split}
\end{equation}
for all $t_a$, $t_b\in[0,\tmax]$ with $t_a\leq t_b$.
Now, we are ready to estimare the quantities at the right hand side of
\eqref{DiscreteMEE_55}. First, we use the H{\"o}lder inequality and
\eqref{GlobalB_1} to arrive at
\begin{equation}\label{GammaB_1}
\begin{split}
\expected\left[\,|\Gamma_1|^{2\kappa}\,\right]
\leq&\,C\,
\sum_{m=0}^{\ssy N-1}\int_{t_m}^{t_{m+1}}
\expected\left[\,|F'(A_m(s))\,U(B_m(s))|^{2\kappa}
\,|Y(s)-\bbg_{m,{\ssy L}}|^{2\kappa}\,\right]\;ds,\\
\leq&\,C\,
\sum_{m=0}^{\ssy N-1}\int_{t_m}^{t_{m+1}}
\left(\expected\left[\,|F'(A_m(s))\,U(B_m(s))|^{4\kappa}\,\right]\right)^{\frac{1}{2}}\,
\left(\expected\left[\,|Y(s)-\bbg_{m,{\ssy L}}|^{4\kappa}\,\right]\right)^{\frac{1}{2}}\;ds,\\
\leq&\,C\,\sum_{m=0}^{\ssy N-1}\int_{t_m}^{t_{m+1}}
\left(\,\expected\left[\,|Y(s)-Y(t_m)|^{4\kappa}
+|Y(t_m)-\bbg_{m,{\ssy L}}|^{4\kappa}
\,\right]\,\right)^{\frac{1}{2}}\;ds,\\
\end{split}
\end{equation}
\begin{equation}\label{GammaB_2}
\begin{split}
\expected\left[\,|\Gamma_2|^{2\kappa}\,\right]
\leq&\,C\,
\sum_{m=0}^{\ssy N-1}\int_{t_m}^{t_{m+1}}
\left(\expected\left[\,|F(A_m(s))\,U'(B_m(s))|^{4\kappa}\,\right]\right)^{\frac{1}{2}}\,
\left(\expected\left[\,|g(s,s)-\bbg_{m,\ell_m}|^{4\kappa}\,\right]\right)^{\frac{1}{2}}\;ds\\
\leq&\,C\,
\sum_{m=0}^{\ssy N-1}\int_{t_m}^{t_{m+1}}
\left(\expected\left[\,|g(s,s)-g(t_m,\tau_{\ell_m})|^{4\kappa}
+|g(t_m,\tau_{\ell_m})-\bbg_{m,\ell_m}|^{4\kappa}\,\right]\right)^{\frac{1}{2}}\;ds,\\
\end{split}
\end{equation}
and
\begin{equation}\label{GammaB_3}
\expected\left[\,|\Gamma_3|^{2\kappa}\,\right]
\leq\,C\,\sum_{m=0}^{\ssy N-1}\int_{t_m}^{t_{m+1}}
|f_{_0}(s)-f_{_0}^m|^{2\kappa}\;ds.
\end{equation}
Next, we combining \eqref{GammaB_2},  \eqref{GammaB_3}, \eqref{DiscreteMC_79}
and \eqref{DiscreteMEE_0}  we obtain 
\begin{equation}\label{GGG1}
\expected\left[\,|\Gamma_2|^{2\kappa}\,\right]
+\expected\left[\,|\Gamma_3|^{2\kappa}\,\right]\leq\,C\,\left[\,(\Delta{t})^{\kappa}
+(\Delta{\tau})^{2\kappa}\,\right].
\end{equation}
Finally, we combine \eqref{GammaB_1}, \eqref{Ymoments} and \eqref{DiscreteMEE_00} to 
obtain
\begin{equation}\label{GGG2}
\expected\left[\,|\Gamma_1|^{2\kappa}\,\right]
\leq\,C\,\left[\,(\Delta{t})^{\kappa}
+(\Delta{\tau})^{2\kappa}\,\right].
\end{equation}
Thus, the error estimate \eqref{DiscreteMEE_000} is a simple consequence of \eqref{DiscreteMEE_55},
\eqref{GGG1} and \eqref{GGG2}.
\par
To derive our fourth error estimate, first we set 
$E:=\Lambda_{\ssy\Psi}(g)-\Lambda_{\ssy\Psi}(\bbg)$,
and then we use the H{\"o}lder inequality to obtain
\begin{equation}\label{GGG5}
\expected\left[\,|E|^{2\kappa}\right]
\leq\,C\,\sqrt{E_{\ssy A}}\,\sqrt{E_{\ssy B}},
\end{equation}
where
\begin{equation*}
\begin{split}
E_{\ssy A}:=&\sum_{\ell={\ell_a}}^{\ssy L-1}\int_{\tau_{\ell}}^{\tau_{\ell+1}}
\expected\left[\,\sup_{\epsilon\in[0,1]}\left|\Psi'\left(f_{_0}(\tau)+\epsilon\,g(\tmax,\tau)
+(1-\epsilon)\,\bbg_{{\ssy N},\ell}\right)\right|^{4\kappa}
\,\right]\;d\tau,\\
E_{\ssy B}:=&\sum_{\ell={\ell_a}}^{\ssy L-1}\int_{\tau_{\ell}}^{\tau_{\ell+1}}
\,\expected\left[|g(\tmax,\tau)
-\bbg_{{\ssy N},\ell}|^{4\kappa}\right]\;d\tau.\\
\end{split}
\end{equation*}
Since $\Psi'$ has polynomial growth, the use
of \eqref{DiscreteM_0} and \eqref {KSI_MaximusM_1} yields that
\begin{equation}\label{GGG6}
E_{\ssy A}\leq\,C.
\end{equation}
Also, using \eqref{LipTau1} and \eqref{DiscreteMEE_0}  we obtain
\begin{equation}\label{GGG7}
E_{\ssy B}\leq\,C\,\left[\,(\Delta{t})^{2\kappa}+(\Delta\tau)^{4\kappa}\,\right].
\end{equation}
Thus, the estimate \eqref{DiscreteMEE_0000} follows after combining \eqref{GGG5}, \eqref{GGG6}
and \eqref{GGG7}.
\par
To obtain our fifth error estimate, first we set $E_{\ssy{\mathcal G}}:={\mathcal G}(g)-{\mathcal G}(\bbg)$ and
use the error bound \eqref{DiscreteMEE_000} to obtain
\begin{equation}\label{FunFun_1}
\expected\left[\,|E_{\ssy{\mathcal G}}|^{2\kappa}\,\right]
\leq\,C\,\left[\,\sqrt{{\mathcal G}_{\ssy A_1}}\,\sqrt{{\mathcal G}_{\ssy A_2}}
+\sqrt{{\mathcal G}_{\ssy B_1}}\,\sqrt{{\mathcal G}_{\ssy B_2}}
+(\Delta{t})^{\kappa}+(\Delta{\tau})^{2\kappa}\,\right],
\end{equation}
where
\begin{equation*}
\begin{gathered}
{\mathcal G}_{\ssy A_1}:=\expected\left[\,\big|G(\Lambda_{\ssy\Psi}(g))\big|^{4\kappa}\,\right],
\quad
{\mathcal G}_{\ssy A_2}:=\expected\left[\,\big|F(Y(\tmax))-F(\bbg_{{\ssy N},{\ssy L}})|^{4\kappa}\,\right]\\
{\mathcal G}_{\ssy B_1}:=\expected\left[\,\big|F(\bbg_{{\ssy N},{\ssy L}})\big|^{4\kappa}\,\right],
\quad
{\mathcal G}_{\ssy B_2}:=\expected\left[\,\big|G(\Lambda_{\ssy\Psi}(g))
-G(\Lambda_{\ssy \Psi}(\bbg))|^{4\kappa}\,\right].\\
\end{gathered}
\end{equation*}
Since $F$ and $G$ have polynomial growth, we combine \eqref{Moments_of_Functional}
and \eqref{DiscreteMEE_00} to get
\begin{equation}\label{FunFun_2}
{\mathcal G}_{\ssy A_1}+{\mathcal G}_{\ssy B_1}\leq\,C.
\end{equation}
Since $F'$ has polynomial growth, we use the mean value theorem, the Cauchy-Schwarz inequality,
\eqref{Moments_of_Functional} and the error bound \eqref{DiscreteMEE_00} to have
\begin{equation}\label{FunFun_3}
\begin{split}
{\mathcal G}_{\ssy A_2}\leq&\,
\left(\expected\left[\,\max_{\epsilon\in[0,1]}\big|F'(\epsilon\,Y(\tmax)
+(1-\epsilon)\,\bbg_{{\ssy N},{\ssy L}})\big|^{8\kappa}\,\right]\right)^{\frac{1}{2}}
\,\left(\expected\left[\,|Y(\tmax)-\bbg_{{\ssy N},{\ssy L}}|^{8\kappa}\right]\right)^{\frac{1}{2}}\\
\leq&\,C\,\left[(\Delta{t})^{2\kappa}+(\Delta\tau)^{4\kappa}\right].\\
\end{split}
\end{equation}
Similarly, since $G'$ has polynomial growth, we use the mean value theorem, the Cauchy-Schwarz inequality,
\eqref{Moments_of_Functional}, and the error bound \eqref{DiscreteMEE_0000} to have
\begin{equation}\label{FunFun_4}
\begin{split}
{\mathcal G}_{\ssy B_2}\leq&\,
\left(\expected\left[\,\max_{\epsilon\in[0,1]}\big|G'\left(\epsilon\,\Lambda_{\ssy\Psi}(g)
+(1-\epsilon)\,\Lambda_{\ssy\Psi}(\bbg)\right)\big|^{8\kappa}\,\right]\right)^{\frac{1}{2}}
\,\left(\expected\left[\,|\Lambda_{\ssy\Psi}(g)-\Lambda_{\ssy\Psi}(\bbg)|^{8\kappa}
\right]\right)^{\frac{1}{2}}\\
\leq&\,C\,\left[(\Delta{t})^{2\kappa}+(\Delta\tau)^{4\kappa}\right].\\
\end{split}
\end{equation}
Thus, the error bound \eqref{DiscreteMEE_Functional} is a simple consequence of
\eqref{FunFun_1}, \eqref{FunFun_2}, \eqref{FunFun_3} and \eqref{FunFun_4}.
\end{proof}
%
%
%
%
%
%
\begin{thm}\label{Theorem_Q}
Let ${\mathcal I}:=\{0,\dots,N\}\times\{0,\dots,L+1\}$,
$(\bbg_{n,\ell})_{(n,\ell)\in{\mathcal I}}$ be the
numerical approximations produced by the \text{\rm(EFD)} or the
\text{\rm (EFM)} method, $\Lambda_{\ssy\Psi}(\bbg)$  be defined as in the Theorem~\ref{A_basic_theorem}
and ${\overline\Lambda}_{\ssy\Psi,Q}(\bbg)$ be the quantity defined by \eqref{eq:2.3a}.
We assume that the quadrature rule $Q$ used in \eqref{eq:2.3a} is of order
$p_{_{\ssy Q}}$, $\Psi\in C^{p_{_Q}}(\rset;\rset)$ 
and $f_{_0}\in C^{p_{_Q}}([0,\taumax];{\mathbb R})$.
Also, we assume that $\Psi$ and all its derivatives up to order $p_{\ssy Q}$,
along with the functions $F$ and $G'$, have polynomial growth.
Then, for $\kappa\in{\mathbb N}$, there exist constants
$C^{\ssy Q, A}_{\kappa}$ and $C^{\ssy Q,B}_{\kappa}$,
independent of the partitions of the intervals $[0,\tmax]$
and $[0,\taumax]$, such that  
\begin{equation}\label{QuadE_0}
\left(\expected\left[\,\left|\Lambda_{\ssy\Psi}(\bbg)
-{\overline\Lambda}_{\ssy\Psi,Q}(\bbg)\right|^{2\kappa}\right]
\right)^{\frac{1}{2\kappa}}\leq\,C_{\kappa}^{\ssy Q, A}
\,(\Delta\tau)^{\,p_{_Q}}
\end{equation}
and
\begin{equation}\label{QuadE_11}
\left(\expected\left[\,\left|{\mathcal G}(\bbg)-{\overline{\mathcal G}}(\bbg)
\right|^{2\kappa}\right]\right)^{\frac{1}{2\kappa}}\leq\,C_{\kappa}^{\ssy Q, B}
\,(\Delta\tau)^{p_{_{Q}}}.
\end{equation}
\end{thm}
%
%
%
\begin{proof}
For $\ell=\ell_a,\dots,L-1$, we set
$v_{\ell}(s):=\Psi(\bbg_{{\ssy N},\ell}+f_{_0}(\tau_{\ell}
+s\,\Delta{\tau}_{\ell}))$ for $s\in[0,1]$.  Since the quadrature rule $Q$ has
order $p_{_Q}$, applying a standard argument from the error analysis
for quadrature rules based on the Taylor formula (see, e.g., \cite{DB}), we obtain
\begin{equation}\label{QuadE_1}
\expected\left[\,\left|\Lambda_{\ssy\Psi}(\bbg)
-{\overline\Lambda}_{\ssy\Psi,Q}(\bbg)\right|^{2\kappa}\,\right]
\leq\,C\,(\Delta\tau)^{2\kappa\,p_{_Q}}\,
\expected\left[\,\max_{\ell_a\leq{\ell}\leq{\ssy L-1}}
\max_{[0,1]}\left|\partial_{s}^{p_{_Q}}v_{\ell}\right|^{2\kappa}\,\right].
\end{equation}
Observing that $\partial_s^{p_{_Q}}v_{\ell}(s)=\sum_{j=0}^{\ssy p_{_Q}}
{p_{_Q} \choose j}\,\Psi^{(j)}(\bbg_{{\ssy N},\ell}+f_{_0}(\tau_{\ell}+s\,\Delta{\tau}_{\ell}))
\,\,f_{_0}^{(p_{_Q}-j)}(\tau_{\ell}+s\,\Delta{\tau}_{\ell}))$, assuming that $\Psi^{(j)}$
has polynomial growth $p_j$ for $j=1,\dots,p_{_Q}$, and using \eqref{DiscreteM_0}, we obtain
\begin{equation}\label{QuadE_2}
\begin{split}
\expected\left[\,\max_{\ell_a\leq{\ell}\leq{\ssy L-1}}\sup_{[0,1]}
\left|\partial_s^{p_{_Q}}v_{\ell}\right|^{2\kappa}\,\right]\leq&\,C\,
\max_{\ell_a\leq{\ell}\leq{\ssy L-1}}\sum_{j=0}^{p_{_Q}}
\left(1+|\bbg_{{\ssy N},\ell}|^{2\kappa p_j}\right)\\
\leq&\,C.\\
\end{split}
\end{equation}
Now, combine \eqref{QuadE_1} and \eqref{QuadE_2} to arrive at \eqref{QuadE_0}. 
\par
Since $F$ and $G'$ have polynomial growth, using the Cauchy-Schwarz inequality,
the mean value theorem, \eqref{DiscreteMEE_00}, \eqref{Moments_of_Functional},
\eqref{QuadE_0} and \eqref{DiscreteMEE_0000}, we obtain
\begin{equation*}
\begin{split}
\expected\left[\,\big|{\mathcal G}(\bbg)-{\overline {\mathcal G}}(\bbg)\big|^{2\kappa}\,\right]
\leq&\,C\,\left(\expected\left[\,\big|F(\bbg_{\ssy N, L})\big|^{4\kappa}\,\right]\right)^{\frac{1}{2}}
\,\left(\expected\left[\,\big|G(\Lambda_{\ssy\Psi}(\bbg))-G({\overline\Lambda}_{\ssy\Psi,Q}(\bbg))\big|^{4\kappa}\right]\right)^{\frac{1}{2}}\\
\leq&\,C\,\left(\expected\left[\,
\sup_{\epsilon\in[0,1]}\big|G'\left(\,\epsilon\,\Lambda_{\ssy\Psi}(\bbg)+(1-\epsilon)\,
{\overline\Lambda}_{\ssy \Psi,Q}(\bbg)\,\right)\big|^{8\kappa}\right]\right)^{\frac{1}{4}}
\,\left(\expected\left[\,\big|\Lambda_{\ssy\Psi}(\bbg)
-{\overline\Lambda}_{\ssy\Psi,Q}(\bbg)\big|^{8\kappa}\right]
\right)^{\frac{1}{4}} \\
\leq&\,C\,(\Delta\tau)^{2\kappa p_{_Q}}
\end{split}
\end{equation*}
which yields the estimate \eqref{QuadE_11}.
\end{proof}
%
%
%
%
%
%
%
%
%
\section{Computable Weak Error Approximation}\label{sec:3}
In this section we present a computable approximation
for the weak $t-$ and $\tau-$ discretization error $E_{\ssy D}$
defined in \eqref{eq:2.5a} for the (EFD) method.
%
%
%
%
%
In Theorem \ref{thm:3.1} below we give an estimate of $E_{\ssy D}$
which, as the step size of both the time and maturity time partitions go
to zero and the number of realizations goes to infinity,
is asymptotically correct.
%
%
On the other hand, the statistical error
$E_{\ssy S}$ can be analyzed by the Central Limit Theorem or
Berry-Esseen Theorem, a standard procedure in Monte Carlo methods
(cf. Section \ref{sec:4}).
While, in Theorem~\ref{Theorem_Q} we have estimated the quadrature
error $E_{\ssy Q}$, concluding that when the order $p_{\ssy Q}$ of the
quadrature rule $Q$ we use in \eqref{eq:2.3a} is sufficiently
large, the quadrature error, $E_{\ssy Q}$, is a higher order term
in the expansion of the computational error. 
%
%
%
%
%
\par
To have an easier access to the results and the techniques of \cite{STZ},
we reformulate problem \eqref{eq:1.5a}-\eqref{eq:1.5b}, letting
the process $g=g(t,\tau)$ be the solution of the problem
\begin{equation}\label{eq:g_proc_def}
\begin{split}
dg(t,\tau)=&\,a(t,\tau,g(t,t))\;dt
+b(t,\tau,g(t,t)){\cdot}dW(t),\quad\forall\,t\in[0,\tmax],\\
g(0,\tau)=&\,0,\\
\end{split}
\end{equation}
for $\tau\in[0,\taumax]$, where
$a:[0,\tmax]\times[0,\taumax]\times\rset\rightarrow\rset$,
$b:[0,\tmax]\times[0,\taumax]\times\rset\rightarrow\rset^{\ssy J}$
given by
\begin{equation*}
\begin{split}
a(t,\tau,x)\equiv&\,\ksi^2(x+f_{_0}(t))\,\tlambda(t,\tau),\\
b(t,\tau,x)\equiv&\,\ksi(x+f_{_0}(t))\,\lambda(t,\tau).\\
\end{split}
\end{equation*}
%
%
%
%
We approximate the unknown process $g(t,\tau)$ by a time and
maturity discretization $\barrg(t,\tau)$,  with
$t\in(t_n)_{n=0}^{\ssy N}$ and
$\tau\in(\tau_{\ell})_{\ell=0}^{\ssy L-1}$, based on the (EFD)
method, which, for $n=0,\dots,N-1$, reads
\begin{equation}\label{eq:3.1}
\aligned \barrg(t_{n+1},\tau_\ell)=&\,\barrg(t_n,\tau_\ell)
+a(t_n,\tau_\ell,\barrg(t_n,\tau_{\ell_n}))\,\Delta t_n+
b(t_n,\tau_\ell,\barrg(t_n,\tau_{\ell_n})){\cdot}\Delta W_n,
\quad\ell=0,\dots,L-1,\\
%
\barrg(0,\tau_\ell)=&\,0,\quad \ell=0,\dots,L.
\endaligned
\end{equation}
For the analysis of the (EFD) method, it is useful to extend its
definition for all times $t$ and all maturities $\tau$ as follows:
for $n=0,\dots,N-1$ and $\ell=0,\dots,L-1$, set
\begin{equation}\label{eq:3.2}
\begin{split}
\barrg(t,\tau)=&\,\barrg(t_n,\tau_{\ell})+
a(t_n,\tau_\ell,\barrg(t_n,\tau_{\ell_n}))(t-{t_n})+
b(t_n,\tau_\ell,\barrg(t_n,\tau_{\ell_n}))\cdot (W(t)-W(t_n))\\
=&\,\barrg(t_n,\tau_{\ell})+\int_{t_n}^t
\barra(s,\tau,\barrg(t_n,\tau_{\ell_n}))ds+
\int_{t_n}^t\barrb(s,\tau,\barrg(t_n,\tau_{\ell_n})){\cdot}dW(s),
\quad\forall\,t\in[t_n,t_{n+1}),\\
\barrg(0,\tau)=&\,0,
%
%
\end{split}
\end{equation}
for $\tau\in[\tau_{\ell},\tau_{\ell+1})$, where $\barra$ and
$\barrb$ are the piecewise constant approximations
\begin{equation}\label{eq:3.3a}
\begin{split}
\barra(t,\tau,x)\left|_{(t,\tau)\in[t_n,t_{n+1})
\times[\tau_{\ell},\tau_{\ell+1})}\right.
\equiv&\,a(t_n,\tau_{\ell},x)=
\ksi^2(x+f_0(t_n))\kenn\tlambda(t_n,\tau_{\ell}),\\
\barrb(t,\tau,x)\left|_{(t,\tau)\in[t_n,t_{n+1})
\times[\tau_{\ell},\tau_{\ell+1})}\right.\equiv&\,b(t_n,\tau_\ell,x)=
\ksi(x+f_0(t_n))\,\lambda(t_n,\tau_{\ell}).\\
\end{split}
\end{equation}
Thus, the extension above results in $\barrg(t,.)\in
S_{\ssy\Delta\tau}$ for any time $t\in [0,\tmax]$.

%
\begin{thm}\label{thm:3.1}
Let ${\mathcal I}:=\{0,\dots,N\}\times\{0,\dots,L-1\}$, $(\bbg_{n,\ell})_{\ssy (n,\ell)\in{\mathcal I}}$
be the numerical approximations produced by the {\rm (EFD)} method. Also, we assume that the
functions $F$, $U$, $\Psi$, $G$ along with their derivatives have polynomial growth.
Also, we set
\begin{equation}\label{eq:3.3b}
d(t,\tau,\widetilde{\tau},x):=\tfrac{1}{2}\,
\ksi^2(x+f_0(t))\,\lambda(t,{\widetilde{\tau}})\cdot\lambda(t,{{\tau}}),
\end{equation}
for  $x\in{\mathbb R}$, $t\in[0,\tmax]$ and $\tau$, ${\widetilde\tau}\in[0,\taumax]$.
Then the computational error of the {\rm(EFD)} method has the expansion
\begin{equation}\label{eq:3.4}
\begin{split}
E_{\ssy D}:=&\,\expected\left[{\mathcal G}(g)\right]-\expected\left[{\mathcal G}(\bbg)\right]
=E_{\ssy D,\rm tau}+ E_{\ssy D,\rm tim} +\OM((\Delta t)^2+(\Delta
\tau)^2),\\
\end{split}
\end{equation}
where
\begin{equation}\label{eq:3.5}
\begin{split}
E_{\ssy D,\rm tau}=&\,\sum_{n=0}^{\ssy N-1}{\Delta t_n}\,\left\{\,
\sum_{\ell=0}^{\ssy L-1}\Delta\tau_\ell\,\expected\left[\,
\tfrac{a(t_{n},\tau_{\ell+1},\barrg(t_{n},\tau_{\ell_n}))
-a(t_{n},\tau_\ell,\barrg(t_{n},\tau_{\ell_n}))}{2}\,\,\varphib_{n,\ell}
\,\right]\,\right\}\\
&\,+\sum_{n=0}^{\ssy N-1}{\Delta t_n}\left\{\,\sum_{\ell=0}^{\ssy
L-1} \sum_{\ell'=0}^{\ssy
L-1}\Delta\tau_\ell\,\Delta{\tau_{\ell'}}\,
\expected\left[\,\tfrac{d(t_{n},\tau_{\ell+1},\tau_{\ell'+1},\barrg(t_n,\tau_{\ell_n}))
-d(t_n,\tau_\ell,\tau_{\ell'},\bbg(t_n,\tau_{\ell_n}))
}{2}\,\,\varphib'_{n,\ell,\ell'}\right]
\,\right\},\\
\end{split}
\end{equation}
and
%
%
\begin{equation}\label{eq:3.6}\aligned
 E_{{\ssy D},\rm tim}=&\,\sum_{n=0}^{\ssy N-1}\tfrac{\Delta t_n}{2}\,\Bigg\{
\expected\left[\,\left(\,F(\bbg_{n+1,{\ssy L}})
\,U({\overline{\overline r}}_{n+1})
-F(\bbg_{n,{\ssy L}})\,U({\overline{\overline r}}_n)\,\right)
\, \varphib_{n+1,{\ssy L+1}}\,\right]\\
&\hskip2.0truecm+\expected\left[\,
\left({\overline{\overline r}}_{n+1}-{\overline{\overline r}}_n\,\right)
\,\varphib_{n+1,{\ssy L}}\right]\\
&\hskip2.0truecm+\sum_{\ell=0}^{\ssy L-1}
\expected\left[\Bigl(a(t_{n+1},\tau_\ell,\barrg(t_{n+1},t_{n+1})) -
a(t_{n},\tau_\ell,\barrg(t_{n},t_n))\Bigr)\varphib_{n+1,\ell}\right]\Bigg\}\\
&+\sum_{n=0}^{\ssy N-1}\tfrac{\Delta{t}_n}{2}\,\Biggl\{\,
\sum_{\ell,\ell'=0}^{\ssy L-1}\expected\left[\Bigl(d(t_{n+1},\tau_\ell,\tau_{\ell'},\barrg(t_{n+1},t_{n+1}))
- d(t_{n},\tau_\ell,\tau_{\ell'},\barrg(t_{n},t_n))
\Bigr)
\varphib'_{n+1,\ell,\ell'}\right]\,\Biggr\}\\
\endaligned
\end{equation}
with 
\begin{equation*}
{\overline{\overline r}}_n:=\bbg(t_n,t_n)+f_{_0}(t_n)=\bbg_{n,\ell_n}+f_{_0}(t_n).
\end{equation*}
The two leading order terms $E_{{\ssy D},\rm tau}$ and $E_{{\ssy D},\rm tim}$
in the right hand side of \eqref{eq:3.4} are in a posteriori form and based on the discrete
duals  $\varphib_{n}\in \rset^{\ssy L+2}$
%
%
and $\varphib'_{n}\in\rset^{\ssy (L+2)\times (L+2)}$
%
%
which are determined as follows.
%
%
First, set
\begin{equation*}
\begin{split}
{\overline\Lambda}'_{{\ssy\Psi,Q},\ell}(\bbg):=&\,
\Delta\tau_{\ell}\,
\sum_{i=1}^{\ssy N_{\ssy Q}}w_{{\ssy Q},i}\,\Psi'\left(\barrg_{{\ssy
N},\ell} +f_{_0}(\tau_{\ell}+s_{{\ssy Q},i}\,\Delta\tau_{\ell})\right),\\
{\overline\Lambda}''_{{\ssy\Psi,Q},\ell}(\bbg):=&\,
\Delta\tau_{\ell}\,
\sum_{i=1}^{\ssy N_{\ssy Q}}w_{{\ssy Q},i}\,\Psi''\left(\barrg_{{\ssy
N},\ell} +f_{_0}(\tau_{\ell}+s_{{\ssy Q},i}\,\Delta\tau_{\ell})\right)\\
\end{split}
\end{equation*}
for $\ell=\ell_a,\dots,L-1$, and
%
%
\begin{equation*}
c_{n,j}(x):=a(t_n,\tau_j,x)\,\Delta t_n
+b(t_n,\tau_j,x)\cdot\Delta W_n
\end{equation*}
for $x\in{\mathbb R}$ and $j=0,\dots,L-1$.
Then, the first dual $\varphib$ is defined by the dual backward
problem with final datum
\begin{equation}\label{eq:3.7a}
\varphib_{{\ssy N},{\ell}}=\left\{
\aligned
&\,0,\hskip4.76truecm\ell=0,\dots,\ell_a-1,\\
%
%
&\,F(\bbg_{\ssy N,L})\,\,G'({\overline\Lambda}_{\ssy\Psi,Q}(\bbg))
\,\,{\overline\Lambda}'_{{\ssy\Psi,Q},\ell}(\bbg),
\quad\ell=\ell_a,\dots,L-1,\\
%
%
&\,F'(\bbg_{\ssy N,L})\,\,G({\overline\Lambda}_{\ssy\Psi,Q}(\bbg)),
\hskip1.85truecm\ell = L,\\
%
%
&\,1,\hskip4.78truecm\ell = L+1,\\
\endaligned \right.
\end{equation}
%
%
%
%
%
and 
\begin{equation}\label{eq:3.7b}
\varphib_{n,\ell}= \left\{ \aligned
&\varphib_{n+1,\ell},
\hskip6.7truecm\ell\in\{0,\dots,L-1\}\backslash\{\ell_n\},\\
&\Delta{t}_n\,\,\varphib_{n+1,L}+\Delta{t}_n\,F(\barrg_{n,{\ssy L}})
\,\,U'\left({\overline{\overline r}}_n\right)
\,\,\varphib_{n+1,{\ssy L+1}}\\
&\quad+\sum_{j=0}^{\ssy L-1}
c'_{n,j}(\barrg_{n,\ell_n})\,\,\varphib_{n+1,j}
+\varphib_{n+1,\ell_n},\hskip2.1truecm\ell=\ell_n\\
%
&\varphib_{n+1,\ssy L}+\Delta t_n\,\,F'(\bbg_{n,{\ssy L}})
\,\,U\left({\overline{\overline r}}_n\right)
\,\,\varphib_{n+1,{\ssy L+1}}
,\hskip1.6truecm\ell =L,\\
%
&\varphib_{n+1,{\ssy L+1}},\hskip6.4truecm\ell=L+1,\\
\endaligned \right.
\end{equation}
for $n=N-1,\dots,0.$
The second dual, $ \varphib'$, has final datum
\begin{equation}\label{eq:3.8a}
\varphib'_{{\ssy N},\ell,\ell'}=\left\{\aligned
&0,\hskip6.6truecm\ell=0,\dots,\ell_a-1,\,\ell'=0,\dots,L+1,\\
&F(\bbg_{\ssy N,L})\,\,G''({\overline\Lambda}_{\ssy\Psi,Q}(\bbg))
\,\,{\overline\Lambda}_{{\ssy\Psi,Q},\ell}(\bbg)
\,\,{\overline\Lambda}_{{\ssy\Psi,Q},\ell'}(\bbg),
\hskip0.5truecm\ell,\ell'\in\{\ell_a,\dots,L-1\},\,\,\ell\not=\ell',\\
 &F(\bbg_{\ssy N,L})\,\left[G''({\overline\Lambda}_{\ssy\Psi,Q}(\bbg))
({\overline\Lambda}'_{{\ssy\Psi,Q},\ell}(\bbg))^2\right.\\
&\hskip2.0truecm\left.+G'({\overline\Lambda}_{\ssy\Psi,Q}(\bbg))
\,\,{\overline\Lambda}''_{{\ssy\Psi,Q},\ell}(\bbg)\right],
\hskip0.9truecm\ell\in\{\ell_a,\dots,L-1\},\,\,\ell'=\ell,\\
&F'(\bbg_{\ssy N,L})\,\,G'({\overline\Lambda}_{\ssy\Psi,Q}(\bbg))
\,\,{\overline\Lambda}'_{{\ssy\Psi,Q},\ell}(\bbg),
\hskip2.1truecm\ell=\ell_a,\dots,L-1,\,\,\ell'=L,\\
&0,\hskip6.6truecm\ell=\ell_a,\dots,L,\,\,\ell'=L+1,\\
&0,\hskip6.6truecm\ell=\ell_a,\dots,L+1,\,\,\ell'=0,\dots,\ell_a-1,\\
&F'(\bbg_{\ssy N,L})\,\,G'({\overline\Lambda}_{\ssy\Psi,Q}(\bbg))
\,\,{\overline\Lambda}'_{{\ssy\Psi,Q},\ell'}(\bbg),
\hskip2.0truecm\ell=L,\,\,\ell'=\ell_a,\dots,L-1,\\
&\,F''(\bbg_{\ssy N,L})\,\,G({\overline\Lambda}_{\ssy\Psi,Q}(\bbg)),
\hskip3.5truecm\ell = L,\,\,\ell'=L,\\
&0,\hskip6.6truecm\ell=L+1,\,\,\ell'=\ell_a,\dots,L+1,\\
\endaligned\right.
\end{equation}
and solves the recursion
\begin{equation}\label{eq:3.8b}
\varphib'_{n,\ell,\ell'}=\left\{\aligned
&\varphib'_{n+1,\ell,\ell'},\hskip7.0truecm
\ell,\ell'\in\{0,\dots,L-1\}\backslash\{\ell_n\},\\
&\sum_{j,p=0}^{\ssy L-1}(\delta_{j,\ell_n}+c_{n,j}'(\bbg_{n,\ell_n}))
\,(\delta_{p,\ell_n}+c_{n,p}'(\bbg_{n,\ell_n}))
\,\varphib'_{n+1,j,p}\\
&+2\,\Delta{t}_n\,\sum_{j=0}^{\ssy L-1}(\delta_{\ell_n,j}+c_{n,j}'(\bbg_{n,\ell_n}))
\,\varphib'_{n+1,j,{\ssy L}}\\
&+2\,\Delta{t}_n\,\sum_{j=0}^{\ssy L-1}(\delta_{\ell_n,j}+c_{n,j}'(\bbg_{n,\ell_n}))
\,F(\bbg_{n,{\ssy L}})\,U'({\overline{\overline r}}_n)
\,\varphib'_{n+1,j,{\ssy L+1}}\\
&+\varphib'_{n+1,{\ssy L},{\ssy L}}\,(\Delta{t}_n)^2
+\varphib'_{n+1,{\ssy L+1},{\ssy L+1}}\,(\Delta{t}_n)^2
\,(F(\bbg_{n,{\ssy L}})\,U'({\overline{\overline r}}_n))^2\\
&+2\,\Delta{t}_n\,F(\bbg_{n,{\ssy L}})\,U'({\overline{\overline r}}_n)
\,\varphib'_{n+1,{\ssy L},{\ssy L+1}}\\
&+\sum_{j=0}^{\ssy L-1}c_{n,j}''(\bbg_{n,\ell_n})\,
\varphib_{n+1,j}+\Delta{t}_n\,F(\bbg_{n,{\ssy L}})\,U''({\overline{\overline r}}_n)
\,\varphib_{n+1,{\ssy L+1}},
\hskip0.5truecm\ell=\ell'=\ell_n,\\
&\sum_{j=0}^{\ssy L-1}(\delta_{j,\ell_n}+c_{n,j}'(\bbg_{n,\ell_n}))\,\varphib'_{n+1,j,\ell'}
+\Delta{t}_n\,\varphib'_{n+1,{\ssy L},\ell'}\\
&+\Delta{t}_n\,F(\bbg_{n,{\ssy L}})\,U'({\overline{\overline r}}_n)\,
\varphib'_{n+1,{\ssy L+1},\ell'},\hskip2.5truecm \ell=\ell_n,\,\ell'\in\{0,\dots,L-1\}
\backslash\{\ell_n\}\\
%
&\varphib'_{n,\ell',\ell},
\hskip6.6truecm \ell\in\{0,\dots,L-1\}\backslash\{\ell_n\},
\,\ell'=\ell_n,\\
\endaligned \right.
\end{equation}
\begin{equation}\label{eq:3.8cc1}
\varphib'_{n,\ell,\ell'}=\left\{\aligned
&\varphib'_{n+1,{\ssy L},\ell'}
+\Delta{t}_n\,F'(\bbg_{n,{\ssy L}})\,U({\overline{\overline r}}_n)
\,\varphib'_{n+1,{\ssy L+1},\ell'},
\quad\ell=L,\,\ell'\in\{0,\dots,L-1\}\backslash\{\ell_n\},\\
&\varphib'_{n,\ell',\ell},
\hskip5.8truecm\ell\in\{0,\dots,L-1\}\backslash\{\ell_n\},\,\ell'=L,\\
\endaligned\right.
\end{equation}
\begin{equation}\label{eq:3.8cc2}
\begin{split}
\varphib'_{n,\ell,\ell'}=&\Delta{t}_n\,F'(\bbg_{n,{\ssy L}})\,U'({\overline{\overline r}}_n)
\,\varphib_{n+1,{\ssy L+1}}
+\sum_{j=0}^{\ssy L-1}(\delta_{j,\ell_n}+c_{n,j}'(\bbg_{n,\ell_n}))
\,\varphib'_{n+1,{\ssy L},j}\\
&+\Delta{t}_n\,F'(\bbg_{n,{\ssy L}})\,U({\overline{\overline r}}_n))
\,\sum_{j=0}^{\ssy L-1}(\delta_{j,\ell_n}+c_{n,j}'(\bbg_{n,\ell_n}))
\,\varphib'_{n+1,{\ssy L+1},j}\\
&+(\Delta{t}_n)^2\,F'(\bbg_{n,{\ssy L}})\,U({\overline{\overline r}}_n))
\,\left[\varphib'_{n+1,{\ssy L+1},{\ssy L}}
+F(\bbg_{n,{\ssy L}})\,U'({\overline{\overline r}}_n)
\,\varphib'_{n+1,{\ssy L+1},{\ssy L+1}}\right]\\
&+\Delta{t}_n\,\left[\varphib'_{n+1,{\ssy L},{\ssy L}}
+F(\bbg_{n,{\ssy L}})\,U'({\overline{\overline r}}_n)
\,\varphib'_{n+1,{\ssy L},{\ssy L+1}}\right],\quad(\ell,\ell')
\in\{(L,\ell_n),(\ell_n,L)\},\\
\end{split}
\end{equation}
\begin{equation}\label{eq:3.8cc4}
\begin{split}
\varphib'_{n,\ell,\ell'}=&\Delta{t}_n\,F''(\bbg_{n,{\ssy L}})\,U({\overline{\overline r}}_n)
\,\varphib_{n+1,{\ssy L+1}}+\varphib'_{n+1,{\ssy L},{\ssy L}}
+2\,\Delta{t}_n\,F'(\bbg_{n,{\ssy L}})\,U({\overline{\overline r}}_n)
\,\varphib'_{n+1,{\ssy L+1},{\ssy L}}\\
&+(\Delta{t}_n)^2\,(F'(\bbg_{n,{\ssy L}})\,U({\overline{\overline r}}_n))^2\,
\varphib'_{n+1,{\ssy L+1},{\ssy L+1}},\quad\ell=\ell'=L,\\
\end{split}
\end{equation}
\begin{equation}\label{eq:3.8cc3}
\varphib'_{n,\ell,\ell'}=\varphib'_{n+1,{\ssy L},{\ssy L+1}}
+\Delta{t}_n\,F'(\bbg_{n,{\ssy L}})\,U({\overline{\overline r}}_n)
\,\varphib'_{n+1,{\ssy L+1},{\ssy L+1}},\quad(\ell,\ell')\in\{(L,L+1),(L+1,L)\},
\end{equation}
and
\begin{equation}\label{eq:3.8cc4}
\varphib'_{n,\ell,\ell'}=\left\{\aligned
&\varphib'_{n+1,{\ssy L+1},\ell'},\quad
\ell=L+1,\,\ell'\in\{0,\dots,L-1,L+1\}\backslash\{\ell_n\},\\
&\sum_{j=0}^{\ssy L-1}(\delta_{j,\ell_n}+c_{n,j}'(\bbg_{n,\ell_n}))
\,\varphib'_{n+1,{\ssy L+1},j}
+\Delta{t}_n\,\varphib'_{n+1,{\ssy L+1},{\ssy L}}\\
&+\Delta{t}_n\,F(\bbg_{n,{\ssy L}})\,U'({\overline{\overline r}}_n)\,
\varphib'_{n+1,{\ssy L+1},{\ssy L+1}},\quad\ell=L+1,\,\ell'=\ell_n,\\
&\varphib'_{n,\ell',\ell},\quad \ell\in\{0,\dots,L-1,L+1\},\,\ell'=L+1.\\
\endaligned\right.
\end{equation}
\end{thm}
%
%
%
\begin{proof}
The proof is an application of Theorem 2.2 in \cite{STZ}. To be
able to split the time and maturity time discretization errors,
introduce the semidiscretized fluxes $\brr{a}$ and $\brr{b}$ that,
for $\tau_\ell\le \tau <\tau_{\ell+1}$, are defined as 
$\bara(t,\tau,x) \equiv\ksi^2(x+f_0(t))\,\tlambda(t,\tau_\ell)$,
$\barb(t,\tau,x) \equiv\ksi(x+f_0(t))\kenn\lambda(t,\tau_\ell)$
and denote by ${\overline g}$ the corresponding semidiscrete
in $\tau$ solution.
As a first step,
 replace the exact solution of \eqref{eq:g_proc_def},
 $g$, by a finite dimensional approximation:
 a piecewise constant $g_*(t,\cdot)$, which
 is an Euler approximation with a much finer discretization, both in
 time $t$ and maturity time $\tau$, than $\barrg$.
 Thus, $g_*$ uses a time grid
 $({\hat t}_n)_{n=0}^{\ssy P}$ much finer than  $(t_n)_{n=0}^{\ssy N}$,
 and a maturity time grid, $({\hat\tau}_\ell)_{\ell=0}^{\ssy M}$ much finer than
 $(\tau_\ell)_{\ell=0}^{\ssy L}$. Consequently, the number of time steps satisfy
 $P>\!\!>N$, $M>\!\!>L$, respectively, and
$ \Delta {\hat t}:=\max_{0\leq{m}\leq{\ssy P-1}}{\hat t}_{m+1}-{\hat t}_m<\!\!<\Delta t$,
$ \Delta {\hat\tau}:=\max_{0\leq{m}\leq{\ssy M-1}}{\hat\tau}_{m+1}-{\hat\tau}_m<\!\!<\Delta \tau.$
In the application of Theorem 2.2 in \cite{STZ}, include the
$\tau$-discretization error terms $a-\bara$, $b-\barb$ as well as
the $t$-discretization terms $\bara-\barra$, $\barb-\barrb$ in the
error expansion, following Lemmata 2.1-2.5 in \cite{STZ}, to
obtain (\ref{eq:3.4}-\ref{eq:3.6}) for $g$ replaced by the
piecewise constant process $g_*$.
 For this purpose, observe that $\bar g$ can be also thought of as a
 piecewise constant function on the finer $\tau$-partition that defines $g_*$.
%
The second step is to let $M,P\to \infty$ and $\Delta {\hat\tau}, \Delta {\hat t}\to 0$,
using
\begin{equation*}\label{eq:3.9}
\left(\expected\left[\max_{[0,\taumax]}|g(t,\cdot)-g_*(t,\cdot)|^2
+\left|\tfrac{\left[g(t,{\hat\tau}_{m+1})-g(t,{\hat\tau}_m)\right]
-\left[g_*(t,{\hat\tau}_{m+1})-g_*(t,{\hat\tau}_m)\right]}
{\hat\tau_{m+1}-{\hat\tau}_m}\right|\right]\right)^{\frac{1}{2}}
=\OM\left(\Delta {\hat\tau}+(\Delta {\hat t})^{\frac{1}{2}}\right),
\end{equation*}
for $t\in[0,\tmax]$ and $m=0,\dots,M-2$,
along with similar estimates for the corresponding dual functions
$\varphib,\varphib',\ldots$, to control the higher
order terms in the error expansion. The latter strong convergence estimates follow
moving along the lines of the analysis of Section~\ref{Section_Strong}.
%
%
%
%
%
%
\end{proof}
%
%
%
%
%
%
\begin{remark}\label{rem:3.1}
In the {\rm (EFD)} method the $\tau$-discretization error of
\eqref{eq:3.4} and \eqref{eq:3.5} can, by \eqref{eq:3.3a},
\eqref{eq:3.3b},
%
be expressed by
\begin{equation}\label{eq:3.10}
\aligned E_{\ssy D,\rm tau}=&\sum_{n=0}^{\ssy N-1}{\Delta t_n}\,\Biggl\{
\sum_{\ell=0}^{\ssy L-1}\Delta\tau_\ell\,
\expected\left[\ksi^2({\overline{\overline r}}_n)\kenn\varphib_{n,\ell}\right]
\,\tfrac{{\widetilde{\lambda}}(t_n,\tau_{\ell+1})
-{\widetilde{\lambda}}(t_n,\tau_{\ell})}{2}\\
&+\tfrac{1}{4}\sum_{\ell,\ell'=0}^{\ssy L-1}
\expected\left[\ksi^2(r_n)\,{\varphib'}_{n,\ell,\ell'}\right]\,\left[
\lambda(t_n,\tau_{\ell+1})\cdot\lambda(t_n,\tau_{\ell'+1})
-\lambda(t_n,\tau_{\ell})\cdot\lambda(t_n,\tau_{\ell'})
\right]\,\Delta\tau_\ell\Delta\tau_{\ell'} \Biggr\}\\
\endaligned
\end{equation}
and the time discretization is
\begin{equation}\label{eq:3.11}
\aligned E_{\ssy D,\rm tim}=\sum_{n=0}^{\ssy N-1}\tfrac{\Delta t_n}{2}\,\Biggl\{&
\expected\left[F(\barrg_{n+1,{\ssy L}})\mcJ({\overline{\overline r}}_{n+1})
-F(\barrg_{n,{\ssy L}})\mcJ({\overline{\overline r}}_{n})\right]
+\expected\left[({\overline{\overline r}}_{n+1}-{\overline{\overline r}}_n)
\,\varphib_{n+1,{\ssy L}}\right]\\
&+\sum_{\ell=0}^{\ssy L-1}\expected\left[(\ksi^2({\overline{\overline r}}_{n+1})
\,\tlambda(t_{n+1},\tau_{\ell})-\ksi^2({\overline{\overline r}}_n)
\,\tlambda(t_{n},\tau_{\ell}))\,\varphib_{n+1,\ell}\right]\\
 &+\tfrac{1}{2}\,\sum_{\ell,\ell'=0}^{\ssy L-1}
\expected\Big[\big(\ksi^2({\overline{\overline r}}_{n+1})\,
\lambda(t_{n+1},\tau_{\ell})\cdot\lambda(t_{n+1},\tau_{\ell'})\\
&\hskip3.0truecm -\ksi^2({\overline{\overline r}}_n)
\,\lambda(t_{n},\tau_{\ell})\cdot\lambda(t_{n},\tau_{\ell'})\big)\,
\varphib'_{n+1,\ell,\ell'}\Big]\Biggr\}\cdot
\endaligned
\end{equation}
In Monte Carlo computations all the expected values in \eqref{eq:3.10} and
\eqref{eq:3.11} are naturally approximated by sample averages.
\end{remark}
\begin{remark}
The analysis of the {\rm (EFE)} method follows a similar line as the
estimates of the {\rm (EFD)} method. The difference lies in the
$\tau$-discretization error, which by virtue of the orthogonality
of both $\tlambda-\prj \tlambda$ and $\lambda-\prj \lambda$ to the
subspace of piecewise constant functions $S_{\ssy\Delta \tau}$,
becomes second order accurate. Therefore, more careful expansions,
including interpolation estimates, need to be carried out in order
to capture the second order contributions from the
$\tau$-discretization.
\end{remark}
\section{ Numerical experiments}\label{sec:4}
In this section we provide numerical evidence for the weak computational
error \eqref{Comp_Error} of the numerical methods defined in
Section~\ref{sec:2} approximating the quantity of interest
$\expected\left[\mcF(f)\right]=\expected\left[{\mathcal G}(g)\right]$
described in (\ref{eq:1.3}-\ref{eq:1.3b}).
In particular, we show results from numerical experiments with examples that
have known exact solution which permit a straightforward derivation of an exact
solution to compare with.
%
%
%
The implementation uses double precision {\tt FORTRAN 77} and
simulates the increments of the $J$ independent Wiener processes
by a double precision modification of the functions {\tt ran1} and
{\tt gasdev} proposed in \cite{Pre}.
The numerical quadrature approximation
${\overline\Lambda}_{\ssy\Psi, Q}(\bbg)$ of $\Lambda_{\ssy\Psi}(\bbg)$ 
in \eqref{eq:2.3a} is done via the use of Simpson's quadrature rule.
For the particular case of the (EFD) method, the estimates for the
computational error developed in Theorem~\ref{thm:3.1} are compared
with the exact computational error. 
The numerical results obtained are in agreement with the theory and the
work to compute these estimates is small.
%
%
%
%
%
%
\subsection{Control of the statistical error}
For $M$ independent samples $\{Y(\omega_j)\}_{j=1}^{\ssy M}$ of a
random variable $Y$, with $\expected\left[\,|Y|^6\,\right] < \infty$,
define the sample average ${\mcA}(Y;M)$ and the sample standard
deviation ${\mcS}(Y;M)$ of $Y$ by
\begin{equation*}\label{eq:sample_avg}
{\mcA}(Y;M) \equiv \tfrac{1}{M}\,{\sum_{j=1}^{\ssy M}}
Y(\omega_j)\quad\mbox{\rm and}\quad{\mcS}(Y;M) \equiv
\left[\,{\mcA}(Y^2;M)-({\mcA}(Y;M))^2\,\right]^{\frac{1}{2}}.
\end{equation*}
Let $\sigma\equiv\sqrt{\expected[|Y-\expected[Y]|^2]}$ and consider the random
variable
\begin{equation*}
Z_{\ssy
M}\equiv\tfrac{\sqrt{M}}{\sigma}\,\left(\mcA(Y;M)-\expected[Y]\right)
\end{equation*}
with cumulative distribution function $F_{\ssy Z_M}(x) \equiv
P(Z_{\ssy M}\leq x)$, for $x\in\rset$. Let
\begin{equation*}
\lambda\equiv
\tfrac{1}{\sigma}\,\left(\,\expected\left[\,\left| Y-\expected[Y] \right|^3
\,\right]\,\right)^{\frac{1}{3}}<\infty,
\end{equation*}
then the Berry-Esseen theorem (cf. \cite{Durr} p. 126), gives the
following estimate in the central limit theorem
\begin{equation*}
\sup_{x\in\rset}|F_{\ssy Z_M}(x)-\Phi(x)| \leq
\tfrac{3}{\sqrt{M}}\,\,\lambda^3
\end{equation*}
for the rate of convergence of $F_{\ssy Z_M}$ to the  distribution
function, $\Phi$, of a normal random variable with mean zero and
variance one, i.e.
\begin{equation*}\label{eq:normal}
\Phi(x)=\tfrac{1}{\sqrt{2\pi}}\,\int_{-\infty}^x
e^{-\tfrac{s^2}{2}}\;ds.
\end{equation*}
Since in the examples below $M$ is sufficiently large, i.e. $M\gg
36\,\lambda^6$, the statistical error
\begin{equation*}
{\mathcal E}_{\ssy S}(Y;M)\equiv \expected[Y]-\mcA(Y;M)
\end{equation*}
satisfies, by the Berry-Esseen theorem, the following probability
approximation
\begin{equation*}
P\left( \left[ |\mcE_{\ssy S}(Y;M)|
\leq{c_{_0}}\,\tfrac{\sigma}{\sqrt{M}}\right]\right) \simeq
\,2\Phi(c_{0})-1.
\end{equation*}
In practice choose some constant $c_{_0}\ge1.65$, so the normal
distribution satisfies
\begin{equation*}
1>2\Phi(c_{_0})-1\ge 0.901
\end{equation*}
and the event
\begin{equation}\label{eq:4.1}
|{\mathcal E}_{\ssy S}(Y;M)|\leq{\mathtt E}_{\ssy S}(Y;M)\equiv
c_{_0}\,\,\tfrac{{\mcS}(Y;M)}{\sqrt{M}}
\end{equation}
has probability close to one, which involves the additional step
to approximate $\sigma$ by $\mcS(Y;M)$, cf. \cite{Fi}. Thus, in
the computations $\mathtt{E}_{\ssy S}(Y;M)$ is a good
approximation of the statistical error $\mcE_{\ssy S}(Y;M)$.
\par
For a given $\tol>0$, the goal is to find $M$ such that
$\ErrS(Y;M)\leq\tol$. The algorithm described in \cite{STZ}
adaptively finds the number of realizations $M$ to compute the
sample average ${\mcA}(Y;M)$ as an approximation to $\expected[Y]$. With
large probability, depending on $c_{_0}$, the statistical error in
the approximation is then bounded by $\tol$. For more details on
the implementation of an adaptive algorithm to control the
statistical error, see \cite{STZ}.
%
%

%
\begin{remark}[Computational cost of the error estimates]\label{rem:3.2}
The work to approximate $\expected\left[{\mathcal G}(g)\right]=\expected[X]$ within an accuracy $\tol$
is $ \OM\left(\tfrac{\text{\rm Var}[X]}{\tol^{4}}\right)$, provided we use the Monte
Carlo version of the EFD method as in \eqref{eq:2.4b}. It is
therefore important to try to use both variance reduction
techniques and adaptive methods to save computational effort. On
the other hand, the work needed to compute sufficiently sharp
error estimates as described in {\rm Theorem~\ref{thm:3.1}} is only
$\OM({\tol^{-3}})$. The number of realizations needed to have a
statistical error in the error bound much smaller than $\tol$ is
only $\OM(\tol^{-1})$ instead of the $\OM\left(\tfrac{\text{\rm Var}[X]}{\tol^2}\right)$
realizations we need to compute an approximation of ${\mathcal F}(g)$
using \eqref{eq:2.4b}, while the work to compute the error
estimate for each realization is still $\OM(\tol^{-2})$, including
the computation of the duals $\varphib$ and $\varphib'$.
This surprising reduction of work for  $\varphib$ and $\varphib'$
is special for the \text{\rm HJM} model studied here. For general
SDEs the corresponding work would be $\OM(\tol^{-4})$ instead of
$\OM(\tol^{-2})$.
 Thus, cheap and sharp error bounds  are obtained
by the use of the a  posteriori error estimates in
{\rm Theorem~\ref{thm:3.1}}.
Observe that if variance reduction techniques are
applied to the approximation of $\expected\left[{\mathcal G}(g)\right]$,
 it is natural to try to use them also
to reduce the variance in the error estimators.
\end{remark}
%
%
%
\begin{remark}[Variance reduction techniques]
The use of variance reduction techniques can decrease
substantially the statistical errors.
In particular the so called {\it antithetic variates} technique
introduced in \cite{HM} reduces the variance in a sample estimator
${\mcA}(M;Y)$ by using another estimator  ${\mcA}(M;{\widetilde Y})$ with the
same expectation as the first one, but which is negatively
correlated with the first. Then, the improved estimator is
${\mcA}(M;\frac{Y+{\widetilde Y}}{2})$. Here, the choice of $Y$ and ${\widetilde Y}$
relates to the Wiener process $W$ and its reflection along the
time axis, $-W$, which is also a Wiener process. If a realization
of the Wiener process, $W(\cdot,\omega_j)$, yields, using one of
the numerical discretizations {\rm(\ref{eq:2.1a}-\ref{eq:2.2b})},
a realization $\bbg(\cdot,\cdot,\omega_j)$ and
$-W(\cdot,\omega_j)$ yields $\bbgg(\cdot,\cdot,\omega_j)$
respectively, then we choose
\begin{equation*}
\tfrac{1}{M}\sum_{j=1}^{\ssy M} \tfrac{F\left(\barrg_{\ssy
N,L}(\omega_j)\right)\, G\left({\overline\Lambda}_{\ssy
\Psi,Q}(\bbg(\cdot,\cdot,\omega_j))\right)+\bbg_{_{\ssy N,L+1}}(\omega_j)
+F\left(\bbgg_{\ssy N,L}(\omega_j)\right)\,
G\left({\overline\Lambda}_{\ssy\Psi,Q}(\bbgg(\cdot,\cdot,\omega_j))\right)
+\bbgg_{_{\ssy N,L+1}}(\omega_j)}{2}
\end{equation*}
as a better estimate. All the numerical results presented below
use antithetic variates.
In general, the use of {\it control variates}, see
\cite{Caflisch}, can be also combined with other variance
reduction methods. For example, the control variates technique is
based on the knowledge of an estimator $Y_{\star}$, positively
correlated with $Y$, whose expected value $\expected[Y_{\star}]$ is known and
relatively close to the desired $\expected[Y]$, yielding
$Y-Y_{\star}+\expected[Y_{\star}]$ as
an improved estimator. The estimates presented in this work do not
preclude the use of control variates, and even though it is not
applied here, it can be a valuable tool in practical computations.
\end{remark}
%
%
%
%
\subsection{Numerical results}
Now let us introduce some notation to be used later
in the description of our numerical results. $\EEtau$ denotes the
sample average approximating the $\tau$-discretization error
\eqref{eq:3.5} and $\EEtim$ denotes the sample average
approximation to the $t$-discretization error \eqref{eq:3.6}.
Beside this, denote by $\EES$ the approximation \eqref{eq:4.1} to
the statistical error $E_{\ssy S}$ introduced in \eqref{eq:2.5a}
and by $\EEtauS$ the approximation \eqref{eq:4.1} to the
statistical error in the estimation of the $\tau$-discretization
error \eqref{eq:3.5} by sample averages. Similarly, $\EEtimS$
denotes the corresponding approximation to the statistical error
in the estimation of the expected values in $t$-discretization
error \eqref{eq:3.6}.
\subsubsection{ Ho-Lee model}
\par The Ho-Lee  model has  $\ksi(x)=\sigma$ and
$\lambda_{_0}(x)=1$ so ${\widetilde\lambda}_{_0}(x)=x$ and
\eqref{eq:1.2a}-\eqref{eq:1.2b} takes the form
\begin{equation}\label{eq:4.2}
\begin{split}
df(t,\tau)=&\,\sigma^2\,(\tau-t)\;dt
+\sigma\,dW(t),\quad 0\leq{t}\leq\tau,\\
f(0,\tau)=&\,f_{_0}(\tau)
\end{split}
\end{equation}
for $\tau\in[0,\taumax]$.
In this example the initial condition is
$f_{_0}(\tau)=r_{_0}-\frac{\sigma^2}{2}\tau^2+\int_0^{\tau}\qhta(s)ds,$
where $r_{_0}$ and  $\sigma$ are real positive constants and
$\qhta:\rset^+\to\rset$ is a given function.
%
%
Then, the exact solution of \eqref{eq:4.2} is
\begin{equation*}
f(t,\tau)=r_{_0}-\tfrac{\sigma^2}{2}(\tau-t)^2
+\int_0^{\tau}\qhta(s)\;ds+\sigma\,W(t),
\quad 0\leq{t}\leq\tau,
\end{equation*}
which follows the normal distribution and therefore, yields bond
prices which are log-normal distributed, allowing the use of Black
and Scholes formulas for the pricing of call and put options on
bonds.
\par
Setting $\tau_a=\tmax$, $F(x)=1-x$, $G(x)=x$, $\Psi(x)=x$ and
$U(x)=0$ in \eqref{eq:1.3a}-\eqref{eq:1.3b},  the functional to be
computed has the form
\begin{equation}\label{eq:4.3}
\expected\left[{\mcF}(f)\right]=\expected\left[\left(1-\int_0^{\tmax}f(s,s)ds\right)
\,\left(\int_{\tmax}^{\taumax}f(\tmax,\tau) \;d\tau\right)\right].
\end{equation}
In the numerical experiments we choose $r_{_0}=0.05$,
$\sigma=0.01$, $\qhta(s)=\frac{1}{10}\,e^{-s}$. Then
$\expected\left[{\mcF}(f)\right]$ is a known function of $\tmax$ and
$\taumax$.
The first experiment sets $\tmax=1.0$ and $\taumax=2.0$, comparing
the efficiency of the (EFD) and (EFE) methods. Table~\ref{tab:4.1}
shows the computational error for both methods and compares the a
posteriori approximation of the error with the true computational
error for the (EFD) method.
Here, a confidence interval for the ratio between the error
approximation and the exact computational error, ${\mcE}_c$,
introduced in \eqref{Comp_Error}, is $[A-B,A+B]$, with
$A\equiv\tfrac{\EEtim+\EEtau}{|{\mcE}_c|}$ and
$B\equiv\tfrac{\EES+\EEtimS+\EEtauS}{|{\mcE}_c|}$. Whenever we use
the (EFD) method we call ${\mathcal E}_{\ssy
C,E\!F\!D}\equiv{\mcE}_c $ and if we use the (EFE) method we call
${\mathcal E}_{\ssy C,E\!F\!E}\equiv{\mcE}_c$.
Observe that the ratio $A\pm B$ of the {\it a posteriori}
approximation of the error over the computational error becomes
closer and closer to one as we refine the time and maturity
partitions, provided that the statistical error is small compared
to the $t$-discretization error and the $\tau$-discretization
error.
In this example, the $t$-discretization gives the largest
contribution to the computational error, and there is no practical
advantage in the use of the (EFE) method.
%
%
%
\begin{table}[ht]\label{tab:4.1}
                   \centering
\begin{tabular}{|c|c|c|c|}
\hline
\ \ \ \ ${{\text{\rm iseed}=-1}}$\ \ \ \ & \ \ \ \ (EFE) \ \ \ \ &  \multicolumn{2}{c|}{(EFD)} \\
\hline
 $N=L$  &  \ \ \ \ ${\mathcal E}_{\ssy C,E\!F\!E}$
 \ \ \ \ & \ \ \ \ ${\mathcal E}_{\ssy C,E\!F\!D}$
 \ \ \ \ &  \ \ \ \ $[A-B,A+B]$ \ \ \ \  \\
\hline
 5  &  $-8.40\times 10^{-4}$      &  $-8.25\times 10^{-4}$      &  $[0.97,0.97]$ \\
\hline
10  &  $-4.16\times 10^{-4}$      &  $-4.08\times 10^{-4}$       & $[0.98,0.99]$ \\
\hline
20  &  $-2.07\times 10^{-4}$      &  $-2.03\times 10^{-4}$       & $[0.98,1.00]$ \\
\hline
\end{tabular}
\par\vskip0.2truecm\par
\centerline{Table~\ref{tab:4.1}. Comparing the (EFD) and (EFE)
methods in the Ho-Lee model approximating} \centerline{ functional
\eqref{eq:4.3} with ${M=5000}$ and $c_0=1.65$.}
\end{table}

\subsubsection{Vasicek model}
The Vasicek model has $\ksi(x)=\sigma$ and
$\lambda_{_0}(x)=e^{-\alpha\,x}$, so
\begin{equation*}
\tlambda_{_0}(x)=\tfrac{1}{\alpha}\,e^{-\alpha\,x}
\,\left(\,1-e^{-\alpha\,x}\,\right)
\end{equation*}
and the forward rate equation (\ref{eq:1.2a}-\ref{eq:1.2b})
becomes
\begin{equation}\label{eq:4.4}
\begin{split}
df(t,\tau)=&\,\tfrac{\sigma^2}{\alpha}\,\left(1-e^{-\alpha
(\tau-t)}\right)\,e^{-\alpha (\tau-t)}\;dt
+\sigma\,e^{-\alpha (\tau-t)}\;dW(t),\quad 0\leq{t}\leq\tau,\\
f(0,\tau)=&\,f_{_0}(\tau)
\end{split}
\end{equation}
for $\tau\in[0,\taumax]$. In this example the initial condition is
\begin{equation*}
f_{_0}(\tau)=\bigl(r_{_0}-\tfrac{\qhta}{\alpha}\bigr)\,
e^{-\alpha\tau}+\tfrac{\qhta}{\alpha}-\tfrac{\sigma^2}{2\,\alpha^2}
\,\left(1-e^{-\alpha\tau}\right)^2,\quad\tau\in[t,\taumax],
\end{equation*}
where $r_{_0}$, $\sigma$, $\alpha$ and $\qhta$ are given positive
constants. The solution of \eqref{eq:4.4} is then
\begin{equation*}
\begin{split}
f(t,\tau)=&\,e^{-\alpha(\tau-t)}\,\left[\, e^{-\alpha
t}\,\left(r_{_0}-\tfrac{\qhta}{\alpha}\right)
+\sigma\,\int_{0}^te^{-\alpha(t-s)}\;dW(s)
\,\right]\\
&\,+\tfrac{\qhta}{\alpha} -\tfrac{\sigma^2}{2\,\alpha^2}
\,\left(1-e^{-\alpha(\tau-t)}\right)^2, \quad 0\leq{t}\leq\tau,
\end{split}
\end{equation*}
which is normally distributed and yields bond prices that are
lognormal, as in the Ho-Lee model.
\par
Here we set $\tau_a=\tmax=0.3$, $\taumax =6.0$, and approximate
again the functional defined in \eqref{eq:4.3}.
In addition, we take $r_{_0}=0.03$, $\alpha=1.0$, $\sigma=0.01$
and $\qhta=0.05$.
%
%
Table~\ref{tab:4.2} displays the computational errors for the
(EFD) and (EFE) methods and compares the {\it a posteriori}
approximation of the error with the true error for the (EFD)
method.
Observe that the ratio $A\pm B$ of the a posteriori approximation
of the error over the computational error becomes closer and
closer to $1$ as we refine the time and maturity partitions,
provided that the statistical error is small compared to the $t$-
and $\tau$-discretization error.
%
%
%
\begin{table}[ht]\label{tab:4.2}
                   \centering
\begin{tabular}{|c|c|c|c|}
\hline
\ \ \ \ ${{\text{\rm iseed}=-1}}$\ \ \ \ & \ \ \ \ (EFE) \ \ \ \
&  \multicolumn{2}{c|}{(EFD)} \\
\hline
 $N=L$  &  \ \ \ \ ${\mathcal E}_{\ssy C,E\!F\!E}$ \ \ \ \
 & \ \ \ \ ${\mathcal E}_{\ssy C,E\!F\!D}$
 \ \ \ \ &  \ \ \ \ $[A-B,A+B]$ \ \ \ \  \\
\hline
 5  &   $-2.30\times 10^{-5}$      &  $-2.07\times 10^{-5}$      &  $[1.92,1.95]$ \\
\hline
10  &   $-2.05\times 10^{-5}$     &  $-1.95\times 10^{-5}$      & $[1.03,1.05]$ \\
\hline
20  &    $-1.06\times 10^{-5}$     &  $-1.00\times 10^{-5}$     & $[0.99,1.02]$ \\
\hline
\end{tabular}
\par\vskip0.2truecm\par
\centerline{Table~\ref{tab:4.2}. Comparing the (EFD) and (EFE)
methods in the Vasicek model approximating}
\centerline{functional \eqref{eq:4.3} with ${M=5000}$ and
$c_0=1.65$.}
\end{table}
\subsubsection{The Cox-Ingersoll-Ross (CIR) model}

Consider the following (CIR) short rate model
\begin{equation}\label{eq:4.5}
r(t)=r_{_0}+\int_0^t(\qhta-\alpha\,r(s))\;ds
+\int_0^t\sigma\sqrt{r(s)}\;dW(s),\ken t\ge0,
\end{equation}
where $\qhta$, $\alpha$ and $\sigma$ are real constants.
To connect the solution $r(t)$ of \eqref{eq:4.5} to the diagonal
value $f(t,t)$ of the solution of an HJM problem, consider, first,
the solution $B=B(t;\tau)$ of the following  Riccati differential
equation (see \cite{BR}):
\begin{equation*}
\begin{split}
\partial_t B(t;\tau)=&\,
\tfrac{1}{2}\,\sigma^2\,B^2(t;\tau)
+\alpha\,B(t;\tau)-1,\quad t\in[0,\tau],\quad\tau\ge0,\\
B(\tau;\tau)=&\,0,
\end{split}
\end{equation*}
which has the form $B(t;\tau)=\psi(\tau-t)$ where
\begin{equation*}
\psi(x)=-\tfrac{\alpha}{\sigma^2}+\tfrac{2}{\sigma^2}\,\,\,\ggo\,\,\,
\tfrac{\sinh(\ggo x)+\tfrac{\alpha}{2\ggo}\cosh(\ggo x)}
{\cosh(\ggo x)+\tfrac{\alpha}{2\ggo}\sinh(\ggo x)}
\quad\text{\rm and}\quad
{\widetilde\gamma}_{_0}:=\tfrac{1}{2}\,\sqrt{2\sigma^2+\alpha^2}.
\end{equation*}
Provided  $\ksi(x)=\sigma\sqrt{\max\{x,0\}}$ and
$\lambda_{_0}(x)=\psi'(x)$, then
${\widetilde\lambda}_{_0}(x)=\psi'(x)\psi(x)$
and the stochastic function
\begin{equation*}
f(t,\tau)=r(t)\,\psi'(\tau-t)+\qhta\,\psi(\tau-t)
\end{equation*}
solves \eqref{eq:1.2a}-\eqref{eq:1.2b} with the initial condition
$f_{_0}(\tau)=r_{_0}\,\psi'(\tau)+\qhta\,\psi(\tau)$.
Taking into account that $\psi'(0)=1$ and $\psi(0)=0$, it follows
that $f(t,t)=r(t)$.
\par
Setting $\tau_a=\tmax$,  $F(x)=e^{-x}$, $G(x)=\max
\left\{e^{-x}-K_{_0},0\right\}$, $\Psi(x)=x$ and  $U(x)=0$ in
(\ref{eq:1.3a}-\ref{eq:1.3b}), the functional to compute in this
example takes the form
\begin{equation}\label{eq:4.6}
\expected\left[\,{\mcF}(f)\,\right]
=\expected\left[\,\exp\left(-\int_0^{\tmax}f(s,s)\;ds\,\right)\,
\max\left\{\exp\left(
-\int_{\tmax}^{\taumax}f(\tmax,\tau)\;d\tau\right)-K_{_0},0\right\}
\right].
\end{equation}
In the numerical experiments we choose $r_{_0}=0.15$,
$\alpha=1.0$, $\sigma=0.1$, $\qhta=0.05$, $\tmax=5.0$,
$\taumax=8.0$ and $K_{_0}=0.5$. Table~\ref{tab:4.3} shows the
computational errors for the (EFD) and (EFE) methods and the ratio
between the approximation of the computational error and the exact
computational error for (EFD) method. There is no practical
difference in this case between the (EFD) and the (EFE) method
since the computational error is mainly $t$-discretization error
and the $\tau$-discretization error is relatively unimportant.
\par
In order to have smooth coefficients in the HJM model
(\ref{eq:1.2a}-\ref{eq:1.2b}) we approximate the function
$\sqrt{\max\{x,0\}}$ in the diffusion term by a Lipschitz function
globally defined in $\rset$ (cf. \cite{Du} p. 252),
\begin{equation*}
\sqrt{\max\{x,0\}}\approx\sqrt{\tfrac{1}{2}(x+\sqrt{x^2+\delta})}
\end{equation*}
where $\delta$ is a small positive constant. Observe that after
this regularization the value of the functional
$\expected\left[\,\mcF(f)\,\right]$ depends on  $\delta$. In the
computations $\delta$  has been taken small enough to make this
dependence negligible with respect to the size of the
computational error.
%
%
%
%
%
%
\begin{table}[ht] \label{tab:4.3}
                   \centering
\begin{tabular}{|c|c|c|c|}
\hline
\ \ \ \ ${{\text{\rm iseed}=-1}}$\ \ \ \ & \ \ \ \ (EFE) \ \ \ \ &  \multicolumn{2}{c|}{(EFD)} \\
\hline
 $N=L$  &  \ \ \ \ ${\mathcal E}_{\ssy C,E\!F\!E}$ \ \ \ \ & \ \ \ \ ${\mathcal E}_{\ssy C,E\!F\!D}$  \ \ \ \
 &  \ \ \ \ $[A-B,A+B]$ \ \ \ \  \\
\hline
 5   &  $1.23\times 10^{-2}$      &  $1.21\times 10^{-2}$      &  $[0.31,0.44]$  \\
\hline
10   &  $5.83\times 10^{-3}$      &  $5.39\times 10^{-3}$      &  $[0.91,0.95]$ \\
\hline
20   &  $2.76\times 10^{-3}$      &  $2.79\times 10^{-3}$      &  $[0.89,0.94]$  \\
\hline
\end{tabular}
\par\vskip0.2truecm\par
\centerline{Table~\ref{tab:4.3}. Comparing the (EFD) and (EFE)
methods in the (CIR) model approximating}
\centerline{functional \eqref{eq:4.6} with ${M=2000}$ and
$c_0=1.65$.}
\end{table}
In this example we compute an accurate numerical approximation of
the exact $\expected\left[\,\mcF(f)\,\right]$ from \eqref{eq:4.6}, via the
Feynman-Kac representation formula, using a numerical solution of
the following backward PDE (cf. \cite{Shr} p. 313),
\begin{equation*}
v_t+(\qhta -\alpha\,r)\,v_r + \tfrac{1}{2}\,\sigma^2\,r\,v_{rr}
-r\,v = 0, \quad t \in [0,\tmax], \quad r\in [0,r_{\max}],
\end{equation*}
with final datum $v(\tmax,r)=\bigl(
B(r,\tmax,\taumax)-K_{_0}\bigr)^+$, where $ B(r,\tmax,\taumax)$
denotes the (CIR) value for a bond with contracting time $\tmax$,
maturity time $\taumax$ and short rate at $\tmax$ equal to $r$. We
also use the boundary conditions
\begin{equation*}
v_t(t,0)+\alpha\,v_r(t,0)=0,\quad v(t,r_{\max})=0,
\end{equation*}
for $ t \in [0,\tmax] $. The value of
$r_{\max}>\!>\tfrac{\qhta}{\alpha}$ is taken sufficiently large so
that the homogeneous Dirichlet boundary at $r = r_{\max}$ has a
negligible effect on the numerical approximation for  $v(0,0.15)=
\expected\left[\,\mcF(f)\,\right]$. The spatial discretization is a
centered finite differences scheme and the time stepping is done
by a diagonally implicit Runge Kutta method, namely the DIRK2
method, see \cite{DV}.
Another way to estimate the exact solution with high accuracy is
to use a formula based on the $\chi^2$ distribution (see
\cite{Reb}, pp. 187-193 for details).
\subsubsection{A two-factor Gaussian model}
A two-factor model has randomness introduced by two scalar
independent Wiener processes $W_1$, $W_2$. In particular, for a
two-factor Gaussian model we have $\ksi(x)=1$,
$\lambda_{_{0,1}}(x)=\sigma_1$ and $\lambda_{_{0,2}}(x)=\sigma_2
e^{-\frac{a_2\,x}{2}}$, where $\sigma_1$, $\sigma_2$ and $a_2$ are
real positive constants. Thus  \eqref{eq:1.2a}-\eqref{eq:1.2b}
takes the form
\begin{equation}\label{eq:4.7}
\begin{split}
df(t,\tau)=&\,\left[(\sigma_1)^2\kenn(\tau-t)+
\tfrac{2(\sigma_2)^2\,e^{-\frac{a_2(\tau-t)}{2}}}{a_2}
\,\left(1-e^{-\frac{a_2(\tau-t)}{2}}\right)
\right]\;dt\\
&\quad\quad+\sigma_1\;dW_1(t)
+\sigma_2\,e^{-\frac{a_2(\tau-t)}{2}}\;dW_2(t),
\quad 0\leq{t}\leq{\tau},\\
f(0,\tau)=&\,f_{_0}(\tau)
\end{split}
\end{equation}
for $\tau\in[0,\taumax]$. Here  the initial condition is
$f_{_0}(\tau)= b_0 + b_1\,e^{-k\,\tau}$
where $b_0$, $b_1$ and $k$ are real constants.
Then, the exact solution of \eqref{eq:4.7} is normal distributed
as in the Ho-Lee and Vasicek models, so explicit formulas are
available for the pricing of put and call options with bonds as
underlyings.
\par
In the numerical experiment we take $\sigma_1=0.02$,
$\sigma_2=0.01$, $a_2 = 0.5$, and compute with the functional
defined in \eqref{eq:4.6} with strike $K_{_0}=0.5$, $\tmax =1$ and
$\taumax =3$. For the initial condition we set $b_0 = 0.0759$,
$b_1 =-0.0439$ and $k = 0.4454$. Table \ref{tab:4.4} shows the
computational errors for the (EFD) and (EFE) methods and the ratio
between the approximation of the computational error and the exact
computational error for method (EFD).
%
\begin{table}[ht]\label{tab:4.4}
                   \centering
\begin{tabular}{|c|c|c|c|}
\hline
\ \ \ \ ${{\text{\rm iseed}=-1}}$\ \ \ \ & \ \ \ \ (EFE) \ \ \ \ &  \multicolumn{2}{c|}{(EFD)} \\
\hline
 $N=L$  &  \ \ \ \ ${\mathcal E}_{\ssy C,E\!F\!E}$ \ \ \ \ & \ \ \ \ ${\mathcal E}_{\ssy C,E\!F\!D}$  \ \ \ \
 &  \ \ \ \ $[A-B,A+B]$ \ \ \ \  \\
\hline
 5   &  $-5.15\times 10^{-4}$
     &  $-6.90\times 10^{-4}$
     &  $[0.98,1.02]$  \\
\hline
10   &   $-2.78\times 10^{-4}$
     &   $-3.50\times 10^{-4}$
     &   $[0.96,1.05]$ \\
\hline
\end{tabular}
\par\vskip0.2truecm\par
\centerline{Table~\ref{tab:4.4}. Comparing the (EFD) and (EFE)
methods in the two-factor Gaussian model }
\centerline{approximating functional \eqref{eq:4.6} with
${M=40000}$ and $c_0=1.65$.}
\end{table}
\section*{Acknowledgements}
This work has been partially supported by:
The Swedish National Network in Applied Mathematics (NTM) `Numerical approximation of stochastic
differential equations' (NADA, KTH),
The EU-TMR project HCL \# ERBFMRXCT960033,
UdelaR and UdeM in Uruguay,
The Swedish Research Council for Engineering Science (TFR) Grant\#222-148,
The VR project `Effektiva numeriska metoder  f\"or stokastiska  differentialekvationer med  till\"ampningar'
(NADA, KTH),
the European Union's Seventh Framework Programme (FP7-REGPOT-2009-1) under grant agreement no. 245749
`Archimedes Center for Modeling, Analysis and Computation' (University of Crete, Greece),
The University of Crete,
and The King Abdullah University of Science and Technology (KAUST).
%
%
%
%
\bibliographystyle{plain}
\end{document}